\documentclass[preprint,3p,11pt,authoryear]{elsarticle}

\usepackage{amssymb,amsthm,mathtools}
\usepackage{lineno}  

\usepackage{enumerate,natbib,geometry}
\usepackage{appendix} 
\usepackage{mathrsfs} 
\usepackage{dsfont} 
\usepackage{cancel} 

\usepackage[colorlinks]{hyperref}

\newtheorem{theorem}{Theorem}[section]

\newtheorem{corollary}[theorem]{Corollary}
\newtheorem{lemma}[theorem]{Lemma}
\newtheorem{remark}[theorem]{Remark}

\makeatletter \@addtoreset{equation}{section} \makeatother

\newcommand{\N}{\mathbb{N}}
\newcommand{\R}{\mathbb{R}}

\newcommand{\PP}{\mathbb{P}}

\newcommand{\oo}{\mathrm{o}}
\newcommand{\leqdef}{\vcentcolon=}

\newcommand{\rd}{{\rm d}}

\allowdisplaybreaks

\begin{document}

\begin{frontmatter}

    \title{An improvement of Tusn\'ady's inequality in the bulk}%

    \author[a1,a2]{Fr\'ed\'eric Ouimet\texorpdfstring{\corref{cor1}\fnref{fn1}}{)}}%

    \address[a1]{McGill University, Montreal, Canada.}%
    \address[a12]{California Institute of Technology, Pasadena, USA.}%

    \cortext[cor1]{Corresponding author}%
    \ead{frederic.ouimet2@mcgill.ca}%

    \fntext[fn1]{F.\ O.\ is supported by a postdoctoral fellowship from the NSERC (PDF) and postdoctoral fellowships from the FRQNT (B3X supplement and B3XR).}%

    \begin{abstract}
        We prove a non-asymptotic generalization of the refined continuity correction developed in \cite{MR538319} for the Binomial distribution, which we then use to improve the versions of Tusn\'ady's inequality from \cite{MR1955348} and \cite{MR2154001} in the bulk.
    \end{abstract}


\end{frontmatter}

\section{Introduction}\label{sec:intro}

    Let $Z\sim \mathrm{Normal}\hspace{0.3mm}(0,1)$, let $\Phi$ denote the cumulative distribution function (c.d.f.) of $Z$, and let $F$ be the c.d.f.\ of the $\mathrm{Binomial}\hspace{0.3mm}(m,1/2)$ distribution.
    If
    \begin{equation}
        Y_m = \frac{m}{2} + \frac{\sqrt{m}}{2} Z\sim \mathrm{Normal}\hspace{0.3mm}(\tfrac{m}{2},\tfrac{m}{4}) \quad \text{and} \quad X_m \leqdef F^{\star}(\Phi(Z))\sim \mathrm{Binomial}\hspace{0.2mm}(m,\tfrac{1}{2}),
    \end{equation}
    where $F^{\star}(p) \leqdef \inf\{x\in \R : F(x) \geq p\}$ denotes the generalized inverse distribution function (or quantile function) associated with $F$, then Tusn\'ady's inequality refers to
    \begin{equation}\label{eq:Tusnady.inequality.original}
        \big|X_m - Y_m\big| \leq 1 + \frac{Z^2}{8}, \quad \text{for all } m\in \N.
    \end{equation}
    This inequality originated in G.\ Tusn\'ady's PhD thesis, i.e., \cite{Tusnady1977phd}, although the author did not provide a complete proof.

    For the uninitiated reader, Tusn\'ady's inequality (also called Tusn\'ady's lemma) lies at the heart of the modern approach to the proof of the Koml\'os-Major-Tusn\'ady (KMT) approximation of the empirical process by a sequence of Brownian bridges on the same probability space, which is one of the most important results of probability theory in the past 50 years.
    The original proof of the KMT approximation was given in \cite{MR375412,MR402883} who applied a diadic scheme together with a somewhat different quantile inequality, whereas \cite[Chapter 4]{MR666546} presented a way to prove the KMT approximation using the diadic scheme with Tusn\'ady's inequality instead (although they didn't prove the inequality either).
    A full proof of Tusn\'ady's inequality first appeared in \cite{MR972783} who used it to improve the rate of convergence of the KMT approximation.
    Throughout the years, many other authors revisited the details of the proof of the KMT approximation such as \cite{MR893903}, \cite{Major_2000_tech_report} and \cite{MR1890338}, some of them (more recently) using some version of Tusn\'ady's inequality such as \cite[Chapter 3]{MR1215046}, \cite{MR1955348}, \cite[Chapter 1]{Dudley_2005_KMT} and \cite{MR2255196}.
    The KMT approximation has been extended to the multivariate setting in \cite{MR996984} and \cite{MR1616527}, to the functional setting in \cite{MR1955345}, and to the dependent setting in \cite{MR3178474}.
    For a review and new results on quantile coupling inequalities and their applications, we refer the reader to \cite{MR3007210} and references therein.

    The proof of Tusn\'ady's inequality proposed in the Appendix of \cite{MR972783} was considered to be highly technical and hard to read by the mathematical community, so it was later simplified in \cite{MR1955348}, who also improved \eqref{eq:Tusnady.inequality.original} as follows:
    \vspace{-1mm}
    \begin{equation}\label{eq:Tusnady.inequality.Massart.2002}
        \big|X_m - Y_m\big| \leq \frac{3}{4} + \frac{Z^2}{8}, \quad \text{for all } m\in \N.
    \end{equation}
    The main idea behind the proof of \eqref{eq:Tusnady.inequality.Massart.2002} is to sum up local comparisons between Binomial and Gaussian probabilities; the methods are elementary but it still requires a careful analysis.

    Around the same time as Massart, \cite{MR2154001} came up with a different approach to improve Tusn\'ady's inequality.
    They applied Laplace's method to the integral representation of the survival function of $X_m\sim \mathrm{Binomial}\hspace{0.2mm}(m,1/2)$, i.e.,
    \begin{equation}
        \PP(X_m\geq k) = \frac{m!}{(k - 1)! (m - k)!} \int_0^{1/2} t^{k-1} (1 - t)^{m-k} \rd t, \quad k\in \{1,2,\dots,m\},
    \end{equation}
    to give yet another version of Tusn\'ady's inequality:
    \begin{equation}\label{eq:Tusnady.inequality.Carter.Pollard.2004}
        \big|X_m - Y_m\big| \leq C + \widetilde{C} \, Z^2 \Big(1 \wedge \frac{|Z|}{\sqrt{m}}\Big), \quad \text{for all } m\in \N,
    \end{equation}
    where $C,\widetilde{C} > 0$ are universal constants assumed large enough but whose explicit form is unknown.
    As pointed out in \cite{MR1955348}, the estimate \eqref{eq:Tusnady.inequality.Carter.Pollard.2004} is not strictly comparable to \eqref{eq:Tusnady.inequality.Massart.2002}: it is better asymptotically but the absolute constants are less precise.
    Massart also explains that improving the constant $3/4$ to $1/2$ is impossible but the optimal constant should be closer to $1/2$ than $3/4$.
    In particular, he conjectured that $C = 1/2 + \text{const.} / \sqrt{m}$ could potentially work by refining the approach in \cite{MR2154001}.

    In this paper, our goal is to improve both \eqref{eq:Tusnady.inequality.Massart.2002} and \eqref{eq:Tusnady.inequality.Carter.Pollard.2004} under the restrictions that $m$ is large enough (with an explicit threshold) and $Z$ falls in the ``bulk'' of the standard normal distribution.
    The precise statement is in the theorem below, where we obtain the aforementioned $1/2$ constant plus an asymptotically negligible residual, but the expansion is also more precise in terms of $Z$.
    Our approach to the proof is outlined in Section~\ref{sec:approach}.

    \begin{theorem}[Tusn\'ady's inequality in the bulk]\label{thm:Tusnady.inequality.improved}
        Let $Z\sim \mathrm{Normal}\hspace{0.3mm}(0,1)$, let $\Phi$ denote the c.d.f.\ of $Z$, and let $F$ denote the c.d.f.\ of the $\mathrm{Binomial}\hspace{0.3mm}(m,1/2)$ distribution.
        Let $m$ be a natural integer satisfying {\color{black} $\sqrt{2\pi} \, 20^6 m^{-1} \leq \sqrt{\log m}$} (for example, $m\geq 4 \cdot 10^7$ works).
        If $Y_m = \frac{m}{2} + \frac{\sqrt{m}}{2} Z\sim \mathrm{Normal}\hspace{0.3mm}(m/2,m/4)$ and $X_m \leqdef F^{\star}(\Phi(Z))\sim \mathrm{Binomial}\hspace{0.2mm}(m,1/2)$, then
        \begin{equation}\label{eq:thm:Tusnady.inequality.improved.eq.1.OG}
            \left|X_m - Y_m - \frac{Z - Z^3}{24 \sqrt{m}}\right| \leq  \frac{1}{2} + \frac{2 \cdot 20^6}{m}, \quad \text{for $|Z| \leq \sqrt{\log m}$}.
        \end{equation}
        In particular, we also have the (weaker) estimate
        \begin{equation}\label{eq:thm:Tusnady.inequality.improved.eq.2.OG}
            \big|X_m - Y_m\big| \leq  \frac{1}{2} + \frac{1}{24} \sqrt{\frac{\log m}{m}} + \frac{1}{24} \cdot \frac{|Z|^3}{\sqrt{m}} + \frac{2 \cdot 20^6}{m}, \quad \text{for $|Z| \leq \sqrt{\log m}$}.
        \end{equation}
    \end{theorem}

    \begin{remark}\label{rem:Tusnady.improvement}
        Under the restriction $\sqrt{2\pi} \, 20^6 m^{-1} \leq \sqrt{\log m}$ and $|Z| \leq \sqrt{\log m}$, Theorem~\ref{thm:Tusnady.inequality.improved} improves on Massart's result \eqref{eq:Tusnady.inequality.Massart.2002} in two ways: asymptotically as $m\to \infty$, and by lowering the constant $3/4$ to $1/2 + 24^{-1} \sqrt{m^{-1} \log m} + 2 \cdot 20^6 m^{-1}$.
        It also improves on Carter and Pollard's result \eqref{eq:Tusnady.inequality.Carter.Pollard.2004} by showing that it holds for the explicit values of $C,\widetilde{C}$:
        \begin{equation}
            C = \frac{1}{2} + \frac{1}{24} \sqrt{\frac{\log m}{m}} + \frac{2 \cdot 20^6}{m} \quad \text{and} \quad \widetilde{C} = \frac{1}{24}.
        \end{equation}
        One instance of the bound \eqref{eq:thm:Tusnady.inequality.improved.eq.1.OG} is illustrated in Figure~\ref{fig:graph.mathematica} for $100 \leq m \leq 1000$.
    \end{remark}

    \begin{figure}
        \centering
        \includegraphics[width=110mm]{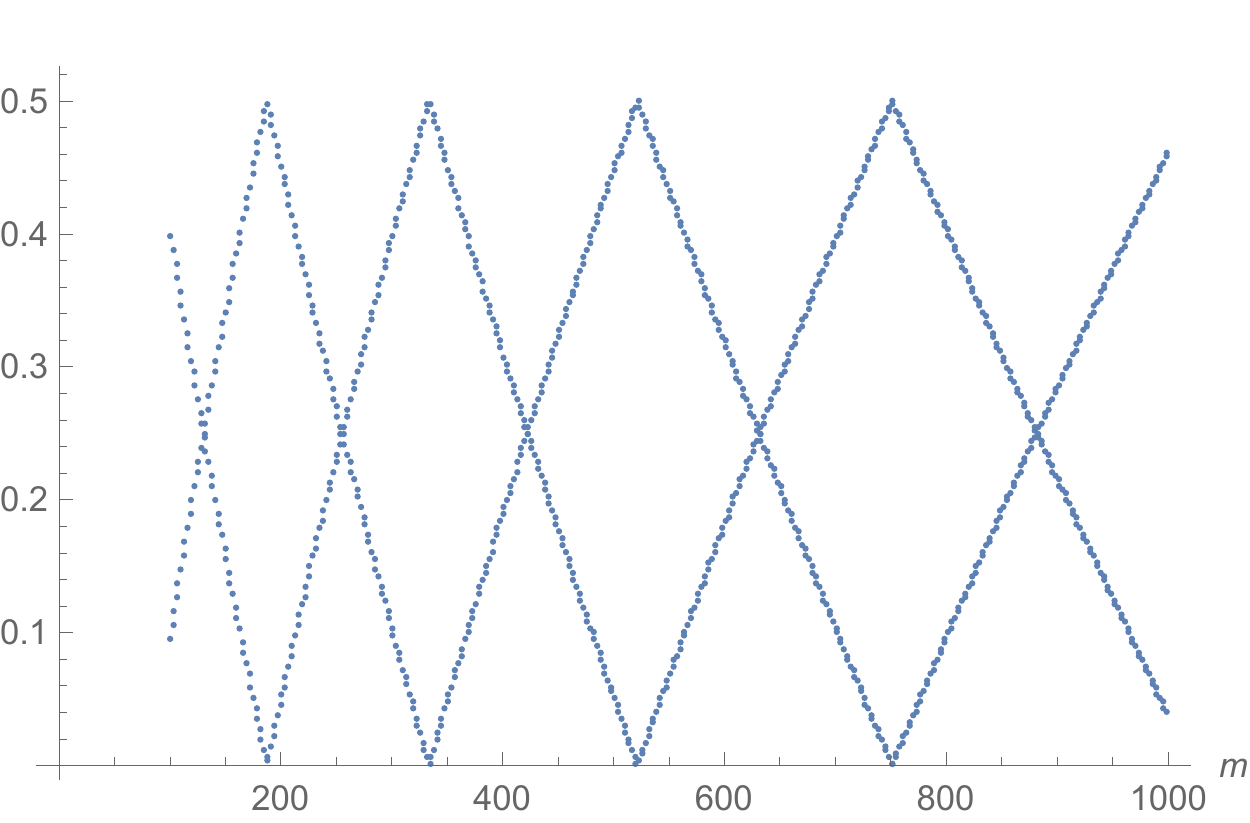}
        \caption{The graph of $\big|X_m(\omega) - Y_m(\omega) - \frac{Z(\omega) - Z^3(\omega)}{24 \sqrt{m}}\big|$ for one instance $\omega$, as a function of $m$.}
        \label{fig:graph.mathematica}
    \end{figure}

\section{Our approach}\label{sec:approach}

    The first main step in our proof of Theorem~\ref{thm:Tusnady.inequality.improved} is to establish a non-asymptotic expansion for the c.d.f.\ and survival function associated with the Binomial probabilities
    \begin{equation}\label{eq:Binomial.pdf}
        P_{k,m}(x) = \binom{m}{k} x^k (1 - x)^{m - k}, \quad k\in \{0,1,\dots,m\},
    \end{equation}
    in terms of the c.d.f.\ and survival function of the normal distribution with the same mean and variance, i.e., $\mathrm{Normal}\hspace{0.3mm}(m x, m x (1 - x))$, the density of which is
    \begin{equation}\label{eq:phi.M}
        y\mapsto \frac{1}{\sqrt{2 \pi} \, \sigma_{m,x}} \exp\Big(-\frac{(y - m x)^2}{2 \sigma_{m,x}^2}\Big), \quad \text{where } \sigma_{m,x}^2 \leqdef m x (1 - x).
    \end{equation}
    Note that \eqref{eq:phi.M} can also be written as $\frac{1}{\sigma_{m,x}} \phi(\delta_y)$ if
    \begin{equation}
        \phi(z) \leqdef \frac{1}{\sqrt{2 \pi}} \exp\Big(-\frac{z^2}{2}\Big), \quad z\in \R, \quad \text{and} \quad \delta_y \leqdef \frac{y - m x}{\sigma_{m,x}}.
    \end{equation}
    More specifically, if $\Phi$ and $\Psi$ denote the c.d.f.\ and the survival function of the standard normal distribution, respectively, then our first goal is to obtain real numbers $c_{m,x}^{\star}(a)$ and $d_{m,x}^{\star}(a)$, as accurate as possible, such that, for some values $a\in \{0,1,\dots,m\}$,
    \begin{equation}\label{eq:non.formal.cdf.approx}
        \sum_{k=0}^a P_{k,m}(x) \approx \Phi(\delta_{a - d_{m,x}^{\star}(a)}) \quad \text{and} \quad \sum_{k=a}^m P_{k,m}(x) \approx \Psi(\delta_{a - c_{m,x}^{\star}(a)}).
    \end{equation}
    These types of results are referred to as continuity corrections in the literature.
    In Theorem~2 of \cite{MR538319}, general expressions for $c_{m,x}^{\star}(a)$ and $d_{m,x}^{\star}(a)$ were found for any fixed $a$, using only elementary methods (Taylor expansions and Stirling's formula), but his results are asymptotic.
    Similarly to \cite{MR1955348}, although the details are very different, the idea of the proof in \cite{MR538319} is to develop a local limit theorem for the Binomial distribution and sum up the approximated probabilities to get an estimate of the survival function in terms of the Gaussian survival function $\Psi$. One important difference is that the errors are additive in \cite{MR538319}, whereas they are multiplicative in \cite{MR1955348} (and \cite{MR2154001}).
    Since Tusn\'ady's inequality is generally valid for finite samples (with possibly an explicit lower bound on the parameter $m$), the results in \cite{MR538319} are not sufficient for our purpose.
    To achieve our first goal, we will therefore prove a non-asymptotic version of \cite[Theorem~2]{MR538319}, using the same elementary methods (see Theorem~\ref{thm:Cressie.non.asymptotic}).
    However, compared to Cressie, we need to find exact bounds for every single error term that arises in the proof, which adds a significant layer of difficulty (this is apparent from the proof).
    Our approximation is also uniform for $a$ in the {\it bulk} of the Binomial distribution and uniform for $x$ in a compact set away from the boundary values $0$ and $1$, namely for
    \begin{align}
        &a\in B_{m,x} \leqdef \bigg\{k\in \{0,1,\dots,m\} :
            \begin{array}{l}
                \max\left\{\left|\frac{\delta_k}{\sqrt{m}} \sqrt{\frac{1 - x}{x}}\right|, \left|\frac{\delta_k}{\sqrt{m}} \sqrt{\frac{x}{1 - x}}\right|\right\} \leq m^{-1/3}
            \end{array}
            \bigg\}, \label{eq:bulk} \\
        &x\in \mathcal{X}_{\tau} \leqdef \bigg\{t\in (0,1) : \max\{t^{-1}, (1 - t)^{-1}\} \leq \tau\bigg\}, \label{eq:X.tau}
    \end{align}
    where $\tau\geq 2$ is a fixed parameter.
    Below is the precise statement:

    \begin{theorem}[Non-asymptotic refined continuity correction]\label{thm:Cressie.non.asymptotic}
        Let $\tau\geq 2$ be given, and let $m$ be an integer satisfying
        \begin{equation}\label{eq:survival.estimate.official.conditions.on.m}
            m\geq 10^3 \vee (2^{3/2} \tau^3) \quad \text{and} \quad 41 \tau^3 \exp\Big(-\frac{m^{1/3}}{8 \tau}\Big) \leq 750 \, 000 \, m^{-4/3}.
        \end{equation}
        Then, uniformly for $a\in B_{m,x}$ and $x\in \mathcal{X}_{\tau}$, we have
        \begin{align}
            &\left|\sum_{k=a}^m P_{k,m}(x) - \Psi(\delta_{a - c_{m,x}^{\star}(a)})\right| \leq \frac{10^6 \tau^5}{m^{3/2}}, \quad \text{if } a \geq m x, \label{eq:thm:Cressie.non.asymptotic.eq.1} \\
            &\left|\sum_{k=0}^a P_{k,m}(x) - \Phi(- \delta_{2 m x - a - c_{m,x}^{\star}(2 m x - a)})\right| \leq \frac{10^6 \tau^5}{m^{3/2}}, \quad \text{if } a \leq m x, \label{eq:thm:Cressie.non.asymptotic.eq.2}
        \end{align}
        where $\sigma_x^2 \leqdef x (1 - x)$ and
        \begin{equation}\label{eq:optimal.choice.c}
            \begin{aligned}
                c_{m,x}^{\star}(a)
                &\leqdef \frac{1}{2} + \frac{(1 - 2x)}{6} \left[\delta_{a - \frac{1}{2}}^2 - 1\right] \\
                &\quad+ \frac{1}{\sqrt{m}} \cdot \frac{1}{\sigma_x} \left\{\Big[\frac{1}{36} - \frac{\sigma_x^2}{36}\Big] \cdot \delta_{a - \frac{1}{2}} + \Big[-\frac{5}{72} + \frac{7 \sigma_x^2}{36}\Big] \cdot \delta_{a - \frac{1}{2}}^3\right\}.
            \end{aligned}
        \end{equation}
        A direct consequence of \eqref{eq:thm:Cressie.non.asymptotic.eq.1} and \eqref{eq:thm:Cressie.non.asymptotic.eq.2} is that, uniformly for $a+1\in B_{m,x}$ and $x\in \mathcal{X}_{\tau}$, we have
        \begin{align}
            &\left|\sum_{k=0}^a P_{k,m}(x) - \Phi(\delta_{a + 1 - c_{m,x}^{\star}(a + 1)})\right| \leq \frac{10^6 \tau^5}{m^{3/2}}, \quad \text{if } a + 1 \geq m x, \label{eq:thm:Cressie.non.asymptotic.eq.2.reflect} \\
            &\left|\sum_{k=a}^m P_{k,m}(x) - \Psi(- \delta_{2 m x - (a + 1) - c_{m,x}^{\star}(2 m x - (a + 1))})\right| \leq \frac{10^6 \tau^5}{m^{3/2}}, \quad \text{if } a + 1 \leq m x. \label{eq:thm:Cressie.non.asymptotic.eq.1.reflect}
        \end{align}
    \end{theorem}

    The proof of Theorem~\ref{thm:Cressie.non.asymptotic} will follow by summing up approximated Binomial probabilities using the local limit theorem we develop in Lemma~\ref{lem:LLT.Binomial}.
    By applying Equations \eqref{eq:thm:Cressie.non.asymptotic.eq.2} and \eqref{eq:thm:Cressie.non.asymptotic.eq.2.reflect} with the specific choices $x = 1/2$, $\tau = 2$ and $a = \frac{m}{2} + \lfloor t \rfloor$, we obtain the corollary below for the $\mathrm{Binomial}\hspace{0.2mm}(m,1/2)$ distribution, which is key for the proof of Theorem~\ref{thm:Tusnady.inequality.improved}.

    \begin{corollary}\label{cor:refined.CC.x.1.2}
        Let $m\geq 10^3$ be an integer, and let $X_m\sim \mathrm{Binomial}\hspace{0.3mm}(m,1/2)$.
        Then, uniformly for $|t|\leq \frac{m^{2/3}}{2} - 1$,
        we have
        \begin{equation}\label{eq:cor:refined.CC.x.1.2}
            \left|\PP(X_m \leq \tfrac{m}{2} + t) - \Phi\bigg(\frac{(1 - \frac{1}{12 m}) s + \frac{1}{3 m^2} s^3}{\sqrt{m} / 2}\bigg)\right| \leq \frac{10^6 2^5}{m^{3/2}},
        \end{equation}
        where $s \leqdef \lfloor t \rfloor + 1/2$.
    \end{corollary}

    In the second main step of our proof, we will control the derivative of $\Phi$ with the mean value theorem and invert the role of $t$ and $\tilde{z} \leqdef \frac{\sqrt{m}}{2} \big(z + \frac{10^6 2^5 m^{-3/2}}{\phi(\sqrt{\log m})}\big) = (1 - \frac{1}{12 m}) s + \frac{1}{3 m^2} s^3$ in \eqref{eq:cor:refined.CC.x.1.2} using a Taylor expansion on the unique real solution of the cubic equation.
    From this, we will be able to deduce
    \begin{equation}
        \Phi(z) \leq \PP\bigg(X_m \leq \frac{m}{2} + \frac{1}{2} + \frac{\sqrt{m}}{2} z + \frac{z - z^3}{24 \sqrt{m}} + \frac{2 \cdot 20^6}{m}\bigg), \quad \text{for } |z| \leq \sqrt{\log m}.
    \end{equation}
    By applying $F^{\star}$ on both sides, we get
    \begin{equation}
        X_m - \frac{m}{2} - \frac{\sqrt{m}}{2} Z - \frac{Z - Z^3}{24 \sqrt{m}}  \leq  \frac{1}{2} + \frac{2 \cdot 20^6}{m}, \quad \text{for } |Z| \leq \sqrt{\log m}.
    \end{equation}
    The statement of Theorem~\ref{thm:Tusnady.inequality.improved} will then follow by the symmetry of $X_m - \frac{m}{2}$ and $Z$.

    \begin{remark}
        As mentioned below \eqref{eq:non.formal.cdf.approx}, the errors throughout the paper are additive instead of multiplicative as in \cite{MR1955348} and \cite{MR2154001}. This is why we only need to control $\Phi'$ (which is simply $\phi$) in the second main step of the proof, instead of $- (\log \Psi)'$ as in \cite{MR2154001}.
        In a sense, we sacrifice a bit on the range of $a$'s for which Theorem~\ref{thm:Cressie.non.asymptotic} holds, but the additivity of the errors makes the second main step of the proof of Theorem~\ref{thm:Tusnady.inequality.improved} much easier, which in turn allows us to be more precise with the expansion in terms of $Z$. This is why we can bring the constant $3/4$ down to $1/2 + \oo(1)$ and get the extra correction factor $\frac{Z - Z^3}{24 \sqrt{m}}$ in \eqref{eq:thm:Tusnady.inequality.improved.eq.1.OG}, which doesn't appear in any other version of Tusn\'ady's inequality.
    \end{remark}

\section{Proofs}\label{sec:proofs}

    First, we prove a uniform and non-asymptotic generalization of the local limit theorem from \cite{MR538319} (i.e., the refined version of Equation (2.4) in that paper).
    We have to be especially careful about the handling of the error terms.

    \begin{lemma}[Local limit theorem]\label{lem:LLT.Binomial}
        Let $\tau\geq 2$ be given and let $m\in \N$ be such that $m \geq 10^3 \vee \tau^{3/2}$.
        Recall the definitions of $B_{m,x}$ and $\mathcal{X}_{\tau}$ from \eqref{eq:bulk} and \eqref{eq:X.tau}, respectively.
        Then, uniformly for $k\in B_{m,x}$ and $x\in \mathcal{X}_{\tau}$, we have
        \begin{align}\label{eq:lem:LLT.Binomial.eq}
            \frac{P_{k,m}(x)}{\frac{1}{\sigma_{m,x}} \phi(\delta_k)} = 1
            &+ m^{-1/2} \cdot \left\{
                \begin{array}{l}
                    -\frac{1}{2} \delta_k \Big[\big(\frac{1 - x}{x}\big)^{1/2} - \big(\frac{x}{1 - x}\big)^{1/2}\Big] \\
                    + \frac{1}{6} \delta_k^3 \Big[x \big(\frac{1 - x}{x}\big)^{3/2} - (1 - x) \big(\frac{x}{1 - x}\big)^{3/2}\Big]
                \end{array}
                \right\} \notag \\
            &\quad+ m^{-1} \cdot \left\{
            \begin{array}{l}
                \frac{1}{8} \delta_k^2 \Big[3 \big(\frac{1 - x}{x}\big) - 2 + 3 \big(\frac{x}{1 - x}\big)\Big] \\[1.5mm]
                -\frac{1}{12} \delta_k^4 \Big[2 x \big(\frac{1 - x}{x}\big)^2 - 1 + 2 (1 - x) \big(\frac{x}{1 - x}\big)^2\Big] \\[1.5mm]
                + \frac{1}{72} \delta_k^6 \Big[x^2 \big(\frac{1 - x}{x}\big)^3 - 2 \sigma_x^2 + (1 - x)^2 \big(\frac{x}{1 - x}\big)^3\Big] \\[1.5mm]
                + \frac{1}{12} \big(1 - \frac{1}{x} - \frac{1}{1 - x}\big)
            \end{array}
            \right\} + E_m,
        \end{align}
        where the error $E_m$ satisfies
        \begin{equation}
            |E_m| \leq m^{-3/2} \cdot 12 \tau^5 (1 + |\delta_k|^{12}).
        \end{equation}
    \end{lemma}

    For good conscience, the terms in the braces in \eqref{eq:lem:LLT.Binomial.eq} coincide exactly with the ones in Result 4.1.10 of \cite{MR207011}, which again is only asymptotic.
    The following remark will be useful to bound some expressions in the proof of Lemma~\ref{lem:LLT.Binomial} and Theorem~\ref{thm:Cressie.non.asymptotic}.
    \begin{remark}\label{rem:x.in.X.tau}
        Notice that $x\in \mathcal{X}_{\tau}$ implies
        \begin{equation}\label{eq:bounds.x.1.minus.x}
            \max\bigg\{\frac{x}{1 - x}, \frac{1 - x}{x}\bigg\} \leq \max\bigg\{\frac{1}{1 - x}, \frac{1}{x}\bigg\} \leq \tau.
        \end{equation}
    \end{remark}

    \begin{proof}[Proof of Lemma~\ref{lem:LLT.Binomial}]
        By taking the logarithm on the left-hand side of \eqref{eq:lem:LLT.Binomial.eq}, we have
        \begin{equation}\label{eq:lem:LLT.Binomial.eq.beginning}
            \begin{aligned}
                \log\bigg(\frac{P_{k,m}(x)}{\frac{1}{\sigma_{m,x}} \phi(\delta_k)}\bigg)
                &= \log \bigg(\frac{\sqrt{2\pi m} \, m!}{k! (m - k)!}\bigg) + (k + \tfrac{1}{2}) \log x \\[-3mm]
                &\quad+ (m - k + \tfrac{1}{2}) \log (1 - x) + \frac{\delta_k^2}{2 \sigma_{m,x}^2}.
            \end{aligned}
        \end{equation}
        From Lindel\"of's estimate of the remainder terms in the expansion of the factorials, found for example on page 67 of \cite{MR1483074}, we know that, for all $n\in \N$,
        \begin{equation}\label{eq:factorial.expansion}
            n! = \sqrt{2\pi} \exp\Big((n + \tfrac{1}{2}) \log n - n + \frac{1}{12 n} + \lambda_n\Big), \quad \text{where } |\lambda_n| \leq \frac{1}{360} n^{-3}.
        \end{equation}
        By applying \eqref{eq:factorial.expansion} in \eqref{eq:lem:LLT.Binomial.eq.beginning} and reorganizing the terms, we get
        \begin{align}\label{eq:big.equation}
            \log\bigg(\frac{P_{k,m}(x)}{\frac{1}{\sigma_{m,x}} \phi(\delta_k)}\bigg)
            &= - (k + \tfrac{1}{2}) \log \bigg(\frac{k}{m x}\bigg) - (m - k + \tfrac{1}{2}) \log \bigg(\frac{m - k}{m (1 - x)}\bigg) \notag \\[-1mm]
            &\quad+ \frac{1}{12 m} \bigg(1 - \frac{1}{x} \cdot \bigg(\frac{k}{m x}\bigg)^{-1} \hspace{-2mm} - \frac{1}{1 - x} \cdot \bigg(\frac{m - k}{m (1 - x)}\bigg)^{-1}\bigg) \notag \\
            &\quad+ \frac{1}{m^3} \left(
                \begin{array}{l}
                    \frac{\lambda_m}{m^{-3}} - \frac{\lambda_k}{k^{-3}} \cdot \frac{1}{x^3} \cdot \big(\frac{k}{m x}\big)^{-3} \\[1mm]
                    - \frac{\lambda_{m-k}}{(m-k)^{-3}} \cdot \frac{1}{(1 - x)^3} \cdot \big(\frac{m - k}{m (1 - x)}\big)^{-3}
                \end{array}
                \right) + \frac{\delta_k^2}{2 \sigma_{m,x}^2}.
        \end{align}
        Since $\frac{k}{m x} = 1 + \frac{\delta_k}{\sqrt{m}} \sqrt{\frac{1 - x}{x}}$ and $\frac{m - k}{m (1 - x)} = 1 - \frac{\delta_k}{\sqrt{m}} \sqrt{\frac{x}{1 - x}}$, the above is
        \begin{align}\label{eq:big.equation.next}
            \log\bigg(\frac{P_{k,m}(x)}{\frac{1}{\sigma_{m,x}} \phi(\delta_k)}\bigg)
            &= - m x \bigg(1 + \frac{\delta_k}{\sqrt{m}} \sqrt{\frac{1 - x}{x}}\bigg) \log \bigg(1 + \frac{\delta_k}{\sqrt{m}} \sqrt{\frac{1 - x}{x}}\bigg) \notag \\
            &\quad- m (1 - x) \bigg(1 - \frac{\delta_k}{\sqrt{m}} \sqrt{\frac{x}{1 - x}}\bigg) \log \bigg(1 - \frac{\delta_k}{\sqrt{m}} \sqrt{\frac{x}{1 - x}}\bigg) \notag \\
            &\quad- \frac{1}{2} \log \bigg(1 + \frac{\delta_k}{\sqrt{m}} \sqrt{\frac{1 - x}{x}}\bigg) - \frac{1}{2} \log \bigg(1 - \frac{\delta_k}{\sqrt{m}} \sqrt{\frac{x}{1 - x}}\bigg) \notag \\
            &\quad+ \frac{1}{12 m} \bigg(1 - \frac{1}{x} \cdot \bigg(1 + \frac{\delta_k}{\sqrt{m}} \sqrt{\frac{1 - x}{x}}\bigg)^{\hspace{-1mm}-1} \hspace{-2mm}- \frac{1}{1 - x} \cdot \bigg(1 - \frac{\delta_k}{\sqrt{m}} \sqrt{\frac{x}{1 - x}}\bigg)^{\hspace{-1mm}-1}\bigg) \notag \\
            &\quad+ \frac{1}{m^3} \left(
                \begin{array}{l}
                    \frac{\lambda_m}{m^{-3}} - \frac{\lambda_k}{k^{-3}} \cdot \frac{1}{x^3} \cdot \Big(1 + \frac{\delta_k}{\sqrt{m}} \sqrt{\frac{1 - x}{x}}\Big)^{-3} \\
                    - \frac{\lambda_{m-k}}{(m-k)^{-3}} \cdot \frac{1}{(1 - x)^3} \cdot \Big(1 - \frac{\delta_k}{\sqrt{m}} \sqrt{\frac{x}{1 - x}}\Big)^{-3}
                \end{array}
                \right) + \frac{\delta_k^2}{2 \sigma_{m,x}^2}.
        \end{align}
        Now, for $|y| \leq m^{-1/3} \leq \frac{1}{10}$, Lagrange error bounds for Taylor expansions imply
        \begin{equation}
            \begin{aligned}
                \left|(1 + y) \log (1 + y) - \left\{y + \frac{y^2}{2} - \frac{y^3}{6} + \frac{y^4}{12}\right\}\right|
                &\leq \left|\frac{-6}{(1 - m^{-1/3})^4}\right| \cdot \left|\frac{y^5}{5!}\right| \leq \frac{|y|^5}{10}, \\
                \left|\log (1 + y) - \left\{y - \frac{y^2}{2}\right\}\right|
                &\leq \left|\frac{2}{(1 - m^{-1/3})^3}\right| \cdot \left|\frac{y^3}{3!}\right| \leq \frac{|y|^3}{2}, \\[0.5mm]
                \left|(1 + y)^{-1} - 1\right|
                &\leq \left|\frac{-1}{(1 - m^{-1/3})^2}\right| \cdot \left|\frac{y}{1!}\right| \leq 2 |y|, \\[0.5mm]
                \left|(1 + y)^{-3}\right|
                &\leq 2.
            \end{aligned}
        \end{equation}
        By applying these estimates in \eqref{eq:big.equation.next} with $y = \frac{\delta_k}{\sqrt{m}} \sqrt{\frac{1 - x}{x}}$ and $y = - \frac{\delta_k}{\sqrt{m}} \sqrt{\frac{x}{1 - x}}$, together with the error bound $|\lambda_n| \leq \frac{1}{360} n^{-3}$ from \eqref{eq:factorial.expansion}, we obtain
        \begin{align}\label{eq:big.equation.2}
            \log\bigg(\frac{P_{k,m}(x)}{\frac{1}{\sigma_{m,x}} \phi(\delta_k)}\bigg)
            &= - m x \left\{
                \begin{array}{l}
                    \cancel{\frac{\delta_k}{\sqrt{m}} \sqrt{\frac{1 - x}{x}}} + \bcancel{\frac{1}{2} \Big(\frac{\delta_k}{\sqrt{m}} \sqrt{\frac{1 - x}{x}}\Big)^2} \\
                    - \frac{1}{6} \Big(\frac{\delta_k}{\sqrt{m}} \sqrt{\frac{1 - x}{x}}\Big)^3 + \frac{1}{12} \Big(\frac{\delta_k}{\sqrt{m}} \sqrt{\frac{1 - x}{x}}\Big)^4
                \end{array}
                \right\} \notag \\
            &\quad- m (1 - x) \left\{
                \begin{array}{l}
                    \cancel{- \frac{\delta_k}{\sqrt{m}} \sqrt{\frac{x}{1 - x}}} + \bcancel{\frac{1}{2} \Big(\frac{\delta_k}{\sqrt{m}} \sqrt{\frac{x}{1 - x}}\Big)^2} \\
                    + \frac{1}{6} \Big(\frac{\delta_k}{\sqrt{m}} \sqrt{\frac{x}{1 - x}}\Big)^3 + \frac{1}{12} \Big(\frac{\delta_k}{\sqrt{m}} \sqrt{\frac{x}{1 - x}}\Big)^4
                \end{array}
                \right\} \notag \\[0.5mm]
            &\quad- \frac{1}{2} \left\{\frac{\delta_k}{\sqrt{m}} \sqrt{\frac{1 - x}{x}} - \frac{1}{2} \Big(\frac{\delta_k}{\sqrt{m}} \sqrt{\frac{1 - x}{x}}\Big)^2\right\} \notag \\
            &\quad- \frac{1}{2} \left\{- \frac{\delta_k}{\sqrt{m}} \sqrt{\frac{x}{1 - x}} - \frac{1}{2} \Big(\frac{\delta_k}{\sqrt{m}} \sqrt{\frac{x}{1 - x}}\Big)^2\right\} \notag \\
            &\quad+ \frac{1}{12 m} \Big(1 - \frac{1}{x} - \frac{1}{1 - x}\Big) + R_{k,m}(x) + \bcancel{\frac{\delta_k^2}{2 \sigma_{m,x}^2}},
        \end{align}
        where
        \begin{align}
            |R_{k,m}(x)|
            &\leq m x \cdot \frac{1}{10} \, \bigg|\frac{\delta_k}{\sqrt{m}} \sqrt{\frac{1 - x}{x}}\bigg|^5 + m (1 - x) \cdot \frac{1}{10} \, \bigg|\frac{\delta_k}{\sqrt{m}} \sqrt{\frac{x}{1 - x}}\bigg|^5 \notag \\
            &\quad+ \frac{1}{2} \cdot \frac{1}{2} \, \bigg|\frac{\delta_k}{\sqrt{m}} \sqrt{\frac{1 - x}{x}}\bigg|^3 + \frac{1}{2} \cdot \frac{1}{2} \, \bigg|\frac{\delta_k}{\sqrt{m}} \sqrt{\frac{x}{1 - x}}\bigg|^3 \notag \\
            &\quad+ \frac{1}{12 m x} \cdot 2 \, \bigg|\frac{\delta_k}{\sqrt{m}} \sqrt{\frac{1 - x}{x}}\bigg| + \frac{1}{12 m (1 - x)} \cdot 2 \, \bigg|\frac{\delta_k}{\sqrt{m}} \sqrt{\frac{x}{1 - x}}\bigg| \notag \\[1mm]
            &\quad+ \frac{1}{m^3} \left\{\frac{1}{360} + \frac{2}{360} \bigg[\frac{1}{x^3} + \frac{1}{(1 - x)^3}\bigg]\right\} \label{eq:bound.on.R.k.m.eq.1} \\
            &\leq m^{-3/2} \cdot \left\{\frac{\tau^{5/2}}{10} |\delta_k|^5 + \frac{\tau^{3/2}}{2} |\delta_k|^3 + \frac{\tau^{3/2}}{3} |\delta_k|\right\} + m^{-3} \cdot \frac{\tau^3}{72}. \label{eq:bound.on.R.k.m.eq.2}
        \end{align}
        Since $k\in B_{m,x}$, $x\in \mathcal{X}_{\tau}$ and $m \geq 10^3 \vee \tau^{3/2}$ (note that $m \geq \tau^{3/2}$ implies $\tau^3 m^{-3} \leq m^{-1}$), an alternative bound for $R_{k,m}(x)$ is (starting from \eqref{eq:bound.on.R.k.m.eq.1}):
        \begin{align}\label{eq:bound.on.R.k.m.simple}
            |R_{k,m}(x)| \leq \frac{m^{-2/3}}{10} + \frac{m^{-1}}{2} + \frac{\tau m^{-4/3}}{3} + \frac{\tau^3 m^{-3}}{72} \leq m^{-2/3}.
        \end{align}
        After the cancellations in \eqref{eq:big.equation.2}, we get
        \begin{align}\label{eq:expression.log.ratio}
            \log\bigg(\frac{P_{k,m}(x)}{\frac{1}{\sigma_{m,x}} \phi(\delta_k)}\bigg)
            &= m^{-1/2} \cdot \left\{
                \begin{array}{l}
                    -\frac{1}{2} \delta_k \Big[\big(\frac{1 - x}{x}\big)^{1/2} - \big(\frac{x}{1 - x}\big)^{1/2}\Big] \\
                    + \frac{1}{6} \delta_k^3 \Big[x \big(\frac{1 - x}{x}\big)^{3/2} - (1 - x) \big(\frac{x}{1 - x}\big)^{3/2}\Big]
                \end{array}
                \right\} \notag \\[1mm]
            &\quad+ m^{-1} \cdot \left\{
                \begin{array}{l}
                    \frac{1}{4} \delta_k^2 \Big[\big(\frac{1 - x}{x}\big) + \big(\frac{x}{1 - x}\big)\Big] \\
                    - \frac{1}{12} \delta_k^4 \Big[x \big(\frac{1 - x}{x}\big)^2 + (1 - x) \big(\frac{x}{1 - x}\big)^2\Big] \\
                    + \frac{1}{12} \big(1 - \frac{1}{x} - \frac{1}{1 - x}\big)
                \end{array}
                \right\} \notag \\[0.5mm]
            &\quad+ R_{k,m}(x).
        \end{align}
        Note that the bound in \eqref{eq:bound.on.R.k.m.simple} together with \eqref{eq:expression.log.ratio} yield (again using $k\in B_{m,x}$, $x\in \mathcal{X}_{\tau}$ and $m \geq 10^3 \vee \tau^{3/2}$):
        \begin{align}\label{eq:expression.log.ratio.bound}
            \left|\log\bigg(\frac{P_{k,m}(x)}{\frac{1}{\sigma_{m,x}} \phi(\delta_k)}\bigg)\right|
            &\leq m^{-1/3} + \frac{1}{6} + \frac{m^{-2/3}}{2} + \frac{m^{-1/3}}{12} + \frac{\tau m^{-1}}{6} + m^{-2/3} \leq \frac{1}{3}.
        \end{align}
        To conclude the proof, we take the exponential on both sides of \eqref{eq:expression.log.ratio}, and we expand the right-hand side with
        \begin{equation}
            \left|e^y - \left\{1 + y + \frac{y^2}{2}\right\}\right| \leq \frac{e^{1/3} |y|^3}{6}, \quad \text{for } |y| \leq \frac{1}{3}.
        \end{equation}
        We obtain
        \begin{align}\label{eq:big.ass.expression}
            \frac{P_{k,m}(x)}{\frac{1}{\sigma_{m,x}} \phi(\delta_k)} = 1
            &+ m^{-1/2} \cdot \left\{
                \begin{array}{l}
                    -\frac{1}{2} \delta_k \Big[\big(\frac{1 - x}{x}\big)^{1/2} - \big(\frac{x}{1 - x}\big)^{1/2}\Big] \\
                    + \frac{1}{6} \delta_k^3 \Big[x \big(\frac{1 - x}{x}\big)^{3/2} - (1 - x) \big(\frac{x}{1 - x}\big)^{3/2}\Big]
                \end{array}
                \right\} \notag \\
            &\quad+ m^{-1} \cdot \left\{
            \begin{array}{l}
                \frac{1}{8} \delta_k^2 \Big[3 \big(\frac{1 - x}{x}\big) - 2 + 3 \big(\frac{x}{1 - x}\big)\Big] \\[1.5mm]
                -\frac{1}{12} \delta_k^4 \Big[2 x \big(\frac{1 - x}{x}\big)^2 - 1 + 2 (1 - x) \big(\frac{x}{1 - x}\big)^2\Big] \\[1.5mm]
                + \frac{1}{72} \delta_k^6 \Big[x^2 \big(\frac{1 - x}{x}\big)^3 - 2 \sigma_x^2 + (1 - x)^2 \big(\frac{x}{1 - x}\big)^3\Big] \\[1.5mm]
                + \frac{1}{12} \big(1 - \frac{1}{x} - \frac{1}{1 - x}\big)
            \end{array}
            \right\} \notag \\
            &\quad+ \Lambda_{k,m}(x),
        \end{align}
        where
        \begin{align}\label{eq:error.Lambda}
            \Lambda_{k,m}(x)
            &\leqdef m^{-3/2} \cdot \left[
                \begin{array}{l}
                    \left\{
                    \begin{array}{l}
                        -\frac{1}{2} \delta_k \Big[\big(\frac{1 - x}{x}\big)^{1/2} - \big(\frac{x}{1 - x}\big)^{1/2}\Big] \\
                        + \frac{1}{6} \delta_k^3 \Big[x \big(\frac{1 - x}{x}\big)^{3/2} - (1 - x) \big(\frac{x}{1 - x}\big)^{3/2}\Big]
                    \end{array}
                    \right\} \\
                    \cdot \left\{
                    \begin{array}{l}
                        \frac{1}{4} \delta_k^2 \Big[\big(\frac{1 - x}{x}\big) + \big(\frac{x}{1 - x}\big)\Big] \\
                        - \frac{1}{12} \delta_k^4 \Big[x \big(\frac{1 - x}{x}\big)^2 + (1 - x) \big(\frac{x}{1 - x}\big)^2\Big] \\
                        + \frac{1}{12} \big(1 - \frac{1}{x} - \frac{1}{1 - x}\big)
                    \end{array}
                    \right\}
                \end{array}
                \right] \notag \\
            &\quad+ m^{-2} \cdot \frac{1}{2} \left\{
                \begin{array}{l}
                    \frac{1}{4} \delta_k^2 \Big[\big(\frac{1 - x}{x}\big) + \big(\frac{x}{1 - x}\big)\Big] \\
                    - \frac{1}{12} \delta_k^4 \Big[x \big(\frac{1 - x}{x}\big)^2 + (1 - x) \big(\frac{x}{1 - x}\big)^2\Big] \\
                    + \frac{1}{12} \big(1 - \frac{1}{x} - \frac{1}{1 - x}\big)
                \end{array}
                \right\}^2 \notag \\
            &\quad+ R_{k,m}(x) + \frac{1}{2} R_{k,m}^2(x) + \tilde{R}_{k,m}(x) \notag \\
            &\quad+ R_{k,m}(x) \cdot \left[
                \begin{array}{l}
                    m^{-1/2} \cdot \left\{
                        \begin{array}{l}
                            -\frac{1}{2} \delta_k \Big[\big(\frac{1 - x}{x}\big)^{1/2} - \big(\frac{x}{1 - x}\big)^{1/2}\Big] \\
                            + \frac{1}{6} \delta_k^3 \Big[x \big(\frac{1 - x}{x}\big)^{3/2} - (1 - x) \big(\frac{x}{1 - x}\big)^{3/2}\Big]
                        \end{array}
                        \right\} \\[1mm]
                    + m^{-1} \cdot \left\{
                        \begin{array}{l}
                            \frac{1}{4} \delta_k^2 \Big[\big(\frac{1 - x}{x}\big) + \big(\frac{x}{1 - x}\big)\Big] \\
                            - \frac{1}{12} \delta_k^4 \Big[x \big(\frac{1 - x}{x}\big)^2 + (1 - x) \big(\frac{x}{1 - x}\big)^2\Big] \\
                            + \frac{1}{12} \big(1 - \frac{1}{x} - \frac{1}{1 - x}\big)
                        \end{array}
                        \right\}
                \end{array}
                \right],
        \end{align}
        and where
        \begin{equation}
            |\tilde{R}_{k,m}(x)| \leq \frac{e^{1/3}}{6} \cdot \left|\log\bigg(\frac{P_{k,m}(x)}{\frac{1}{\sigma_{m,x}} \phi(\delta_k)}\bigg)\right|^3.
        \end{equation}
        By applying Jensen's inequality twice in a row and the bound on $|R_{k,m}(x)|$ from \eqref{eq:bound.on.R.k.m.simple}, we have
        \begin{align}
            |\tilde{R}_{k,m}(x)|
            &\stackrel{\eqref{eq:expression.log.ratio}}{\leq} \frac{e^{1/3}}{6} \cdot \left|
                    \begin{array}{l}
                        m^{-1/2} \cdot \left\{
                            \begin{array}{l}
                                -\frac{1}{2} \delta_k \Big[\big(\frac{1 - x}{x}\big)^{1/2} - \big(\frac{x}{1 - x}\big)^{1/2}\Big] \\
                                + \frac{1}{6} \delta_k^3 \Big[x \big(\frac{1 - x}{x}\big)^{3/2} - (1 - x) \big(\frac{x}{1 - x}\big)^{3/2}\Big]
                            \end{array}
                            \right\} \\[1mm]
                        + m^{-1} \cdot \left\{
                            \begin{array}{l}
                                \frac{1}{4} \delta_k^2 \Big[\big(\frac{1 - x}{x}\big) + \big(\frac{x}{1 - x}\big)\Big] \\
                                - \frac{1}{12} \delta_k^4 \Big[x \big(\frac{1 - x}{x}\big)^2 + (1 - x) \big(\frac{x}{1 - x}\big)^2\Big] \\
                                + \frac{1}{12} \big(1 - \frac{1}{x} - \frac{1}{1 - x}\big)
                            \end{array}
                            \right\} \\
                        + R_{k,m}(x)
                    \end{array}
                \right|^3 \notag \\[2mm]
            &\stackrel{\text{Jensen}}{\leq} \frac{e^{1/3}}{6} \cdot 3^2 \cdot \left[
                \begin{array}{l}
                    m^{-3/2} \cdot 2^2 \cdot \left\{\tau^{3/2} |\delta_k|^3 + \frac{\tau^{9/2}}{6^3} |\delta_k|^9\right\} \\
                    + m^{-3} \cdot 3^2 \cdot \left\{\frac{\tau^3}{2^3} |\delta_k|^6 + \frac{\tau^6}{12^3} |\delta_k|^{12} + \frac{\tau^3}{4^3}\right\} \\
                    + (m^{-2/3})^3
                \end{array}
                \right].
        \end{align}
        To control the last term in \eqref{eq:error.Lambda}, note that
        \begin{equation}
            \begin{aligned}
            &\left|
                \begin{array}{l}
                    m^{-1/2} \cdot \left\{
                        \begin{array}{l}
                            -\frac{1}{2} \delta_k \Big[\big(\frac{1 - x}{x}\big)^{1/2} - \big(\frac{x}{1 - x}\big)^{1/2}\Big] \\
                            + \frac{1}{6} \delta_k^3 \Big[x \big(\frac{1 - x}{x}\big)^{3/2} - (1 - x) \big(\frac{x}{1 - x}\big)^{3/2}\Big]
                        \end{array}
                        \right\} \\[1mm]
                    + m^{-1} \cdot \left\{
                        \begin{array}{l}
                            \frac{1}{4} \delta_k^2 \Big[\big(\frac{1 - x}{x}\big) + \big(\frac{x}{1 - x}\big)\Big] \\
                            - \frac{1}{12} \delta_k^4 \Big[x \big(\frac{1 - x}{x}\big)^2 + (1 - x) \big(\frac{x}{1 - x}\big)^2\Big] \\
                            + \frac{1}{12} \big(1 - \frac{1}{x} - \frac{1}{1 - x}\big)
                        \end{array}
                        \right\}
                    \end{array}
                    \right| \\
                &\leq m^{-1/2} \cdot \bigg\{\tau^{1/2} |\delta_k| + \frac{\tau^{3/2}}{6} |\delta_k|^3\bigg\} + m^{-1} \cdot \left\{\frac{\tau}{2} |\delta_k|^2 + \frac{\tau^2}{12} |\delta_k|^4 + \frac{\tau}{4}\right\}.
            \end{aligned}
        \end{equation}
        Putting all the work above together, we obtain the following bound on $|\Lambda_{k,m}(x)|$ in \eqref{eq:error.Lambda}:
        \begin{align}\label{eq:complete.error.bound}
            |\Lambda_{k,m}(x)|
            &\leq m^{-3/2} \cdot \left[
                \begin{array}{l}
                    \left\{\tau^{1/2} |\delta_k| + \frac{\tau^{3/2}}{6} |\delta_k|^3\right\} \\
                    \cdot \left\{\frac{\tau}{2} |\delta_k|^2 + \frac{\tau^2}{12} |\delta_k|^4 + \frac{\tau}{4}\right\}
                \end{array}
                \right] + m^{-2} \cdot \frac{1}{2} \left[\frac{\tau}{2} |\delta_k|^2 + \frac{\tau^2}{12} |\delta_k|^4 + \frac{\tau}{4}\right]^2 \notag \\
            &\quad+ \left[
                \begin{array}{l}
                    m^{-3/2} \cdot \left\{\frac{\tau^{5/2}}{10} |\delta_k|^5 + \frac{\tau^{3/2}}{2} |\delta_k|^3 + \frac{\tau^{3/2}}{3} |\delta_k|\right\} + m^{-3} \cdot \frac{\tau^3}{72}
                \end{array}
                \right] \notag \\
            &\quad+ \frac{1}{2} \left[
                \begin{array}{l}
                    m^{-3/2} \cdot \left\{\frac{\tau^{5/2}}{10} |\delta_k|^5 + \frac{\tau^{3/2}}{2} |\delta_k|^3 + \frac{\tau^{3/2}}{3} |\delta_k|\right\} + m^{-3} \cdot \frac{\tau^3}{72}
                \end{array}
                \right]^2 \notag \\
            &\quad+ \frac{e^{1/3}}{6} \cdot 3^2 \cdot \left[
                \begin{array}{l}
                    m^{-3/2} \cdot 2^2 \cdot \left\{\tau^{3/2} |\delta_k|^3 + \frac{\tau^{9/2}}{6^3} |\delta_k|^9\right\} \\
                    + m^{-3} \cdot 3^2 \cdot \left\{\frac{\tau^3}{2^3} |\delta_k|^6 + \frac{\tau^6}{12^3} |\delta_k|^{12} + \frac{\tau^3}{4^3}\right\} \\
                    + m^{-2}
                \end{array}
                \right] \notag \\
            &\quad+ \left[
                \begin{array}{l}
                    m^{-3/2} \cdot \left\{\frac{\tau^{5/2}}{10} |\delta_k|^5 + \frac{\tau^{3/2}}{2} |\delta_k|^3 + \frac{\tau^{3/2}}{3} |\delta_k|\right\} + m^{-3} \cdot \frac{\tau^3}{72}
                \end{array}
                \right] \notag \\
            &\qquad\quad\cdot \left[
                \begin{array}{l}
                    m^{-1/2} \cdot \left\{\tau^{1/2} |\delta_k| + \frac{\tau^{3/2}}{6} |\delta_k|^3\right\} \\[2mm]
                    + m^{-1} \cdot \left\{\frac{\tau}{2} |\delta_k|^2 + \frac{\tau^2}{12} |\delta_k|^4 + \frac{\tau}{4}\right\}
                \end{array}
                \right].
        \end{align}
        For $m \geq 10^3 \vee \tau^{3/2}$ and $\tau\geq 2$, the last equation becomes
        \begin{align}
            |\Lambda_{k,m}(x)|
            &\leq m^{-3/2} \cdot \left[\left\{\tau^{3/2} (|\delta_k| + \frac{1}{6} |\delta_k|^3)\right\} \cdot \left\{\tau^2 (1 + |\delta_k|^4)\right\}\right] + m^{-2} \cdot \left[\tau^4 (1 + |\delta_k|^8)\right] \notag \\
            &\quad+ \left[m^{-3/2} \cdot \tau^{5/2} (1 + |\delta_k|^5)\right] + \frac{1}{2} \left[m^{-3/2} \cdot \tau^{5/2} (1 + |\delta_k|^5)\right]^2 \notag \\
            &\quad+ \frac{3 e^{1/3}}{2} \cdot \left[m^{-3/2} \cdot \tau^{9/2} (1 + |\delta_k|^{12})\right] \notag \\
            &\quad+ \left[m^{-3/2} \cdot \tau^{5/2} (1 + |\delta_k|^5)\right] \cdot \left[m^{-1/2} \cdot 2 \tau^2 (1 + |\delta_k|^4)\right] \notag \\
            &\leq m^{-3/2} \cdot 12 \tau^5 (1 + |\delta_k|^{12}).
        \end{align}
        This ends the proof.
    \end{proof}

    By summing up the Binomial probabilities using the local limit theorem (Lemma~\ref{lem:LLT.Binomial}), we can now approximate the survival function of the Binomial distribution.

    \begin{proof}[Proof of Theorem~\ref{thm:Cressie.non.asymptotic}]
        By symmetry, it suffices to prove \eqref{eq:thm:Cressie.non.asymptotic.eq.1}.
        Let $c\in \R$ be a parameter to be chosen later.
        If $k^{\star} \leqdef \max_{k\in B_{m,x}} k$ denotes the largest integer in $B_{m,x}$ and we decompose the interval $[a - \tfrac{1}{2}, k^{\star} + \tfrac{1}{2}]$ into small subintervals of length $1$, we get
        \begin{equation}\label{eq:Cressie.generalization.eq.3}
            \begin{aligned}
                &\sum_{k=a}^m P_{k,m}(x) - \Psi(\delta_{a - c}) \\[-2mm]
                &\qquad= \sum_{\substack{a \leq k \leq m \\ k\in B_{m,x}}} \Big[P_{k,m}(x) - \int_{\delta_{k - \frac{1}{2}}}^{\delta_{k + \frac{1}{2}}} \phi(y) \rd y\Big] - \int_{\delta_{a - c}}^{\delta_{a - \frac{1}{2}}} \phi(y) \rd y + \Psi(\delta_{k^{\star} + \frac{1}{2}}).
            \end{aligned}
        \end{equation}
        Since $k^{\star} + 1$ is outside $B_{m,x}$, note that
        \begin{align}\label{eq:bound.outside.bulk.extreme}
            \delta_{k^{\star} + \frac{1}{2}}
            &= \delta_{k^{\star} + 1} - \frac{1}{2 \sigma_{m,x}} \notag \\[-0.5mm]
            &\geq m^{1/6} \bigg(\max\Big\{\sqrt{\frac{1 - x}{x}}, \sqrt{\frac{x}{1 - x}}\Big\}\bigg)^{-1} - \frac{1}{2 \sigma_{m,x}} \qquad (\text{by } \eqref{eq:bulk}) \notag \\[-0.5mm]
            &\geq \frac{m^{1/6}}{\tau^{1/2}} - \frac{\tau}{2 m^{1/2}} \qquad (\text{by Remark~\ref{rem:x.in.X.tau}}) \notag \\[0.5mm]
            &\geq \frac{3 m^{1/6}}{4 \tau^{1/2}}, \qquad \big(\text{because } m\geq 2^{3/2} \tau^3 ~\text{implies } m^{1/2} \geq 2 \tau^{3/2} m^{-1/6}\big)
        \end{align}
        and similarly,
        \begin{align}\label{eq:bound.outside.bulk.extreme.k.star.alone}
            \delta_{k^{\star}} = \delta_{k^{\star} + 1} - \frac{1}{\sigma_{m,x}} \geq \frac{m^{1/6}}{\tau^{1/2}} - \frac{\tau}{m^{1/2}} \geq \frac{m^{1/6}}{2 \tau^{1/2}}.
        \end{align}
        By the well-known Mills ratio inequality, $\Psi(x) \leq x^{-1} \phi(x)$ for all $x > 0$, we deduce from \eqref{eq:bound.outside.bulk.extreme} that
        \begin{equation}\label{eq:main.decomposition.last.term.bound}
            \Psi(\delta_{k^{\star} + \frac{1}{2}}) \leq \sqrt{\frac{8 \tau}{9 \pi m^{1/3}}} \exp\Big(-\frac{9 m^{1/3}}{32 \tau}\Big) \quad \text{in } \eqref{eq:Cressie.generalization.eq.3},
        \end{equation}
        and from \eqref{eq:bound.outside.bulk.extreme.k.star.alone} that
        \begin{equation}\label{eq:main.decomposition.last.term.bound.k.star.alone}
            \begin{aligned}
                &\phi(\delta_{k^{\star}}) \leq \frac{1}{\sqrt{2\pi}} \exp\Big(-\frac{m^{1/3}}{8 \tau}\Big), \\
                &\Psi(\delta_{k^{\star}}) \leq \sqrt{\frac{2 \tau}{\pi m^{1/3}}} \exp\Big(-\frac{m^{1/3}}{8 \tau}\Big).
            \end{aligned}
        \end{equation}
        The Taylor expansion of $\phi(y)$ around any $y_0\in \R$ is, for all $y$'s such that $|y - y_0| \leq \varepsilon$,
        \begin{equation}\label{eq:Taylor.phi.Sigma}
            \begin{aligned}
                \phi(y)
                &= \phi(y_0) + \phi'(y_0) (y - y_0) + \tfrac{1}{2} \phi''(y_0) (y - y_0)^2 \\[1mm]
                &\qquad+ \tfrac{1}{6} \phi'''(y_0) (y - y_0)^3 + A_m,
            \end{aligned}
        \end{equation}
        where
        \begin{align}
            |A_m|
            &\leq \frac{\max_{|z - y_0| \leq \varepsilon} |z^4 - 6 z^2 + 3| \, \phi(z)}{24} \, |y - y_0|^4.
        \end{align}
        By taking $y_0 = \delta_k$ in \eqref{eq:Taylor.phi.Sigma} and by integrating on the compact interval $[\delta_{k - \frac{1}{2}},\delta_{k + \frac{1}{2}}]$, the first and third order terms disappear because of the symmetry, so we have
        \begin{equation}\label{eq:Cressie.generalization.eq.1}
            \begin{aligned}
                &\left|\int_{\delta_{k - \frac{1}{2}}}^{\delta_{k + \frac{1}{2}}} \phi(y) \rd y - \left\{\frac{\phi(\delta_{k})}{\sigma_{m,x}} + 0 + \frac{1}{2} \phi''(\delta_k) \cdot \int_{-\frac{1}{2\sigma_{m,x}}}^{\frac{1}{2\sigma_{m,x}}} y^2 \rd y + 0\right\}\right| \\
                &\leq \frac{\max_{z\in [\delta_{k - 1/2}, \delta_{k + 1/2}]} |z^4 - 6 z^2 + 3| \, \phi(z)}{24} \int_{-\frac{1}{2\sigma_{m,x}}}^{\frac{1}{2\sigma_{m,x}}} y^4 \rd y.
            \end{aligned}
        \end{equation}
        In other words,
        \begin{equation}\label{eq:Cressie.generalization.eq.1.restated}
            \begin{aligned}
                &\left|\int_{\delta_{k - \frac{1}{2}}}^{\delta_{k + \frac{1}{2}}} \phi(y) \rd y - \frac{\phi(\delta_{k})}{\sigma_{m,x}} \left\{1 + \frac{(\delta_k^2 - 1)}{24 \sigma_{m,x}^2}\right\}\right| \\
                &\leq \frac{\max_{z\in [\delta_{k - 1/2}, \delta_{k + 1/2}]} |z^4 - 6 z^2 + 3| \, \phi(z)}{1920 \, \sigma_{m,x}^5}.
            \end{aligned}
        \end{equation}
        Similarly, it is easily shown that
        \begin{equation}\label{eq:Cressie.generalization.eq.2}
            \begin{aligned}
                &\left|\int_{\delta_{a - c}}^{\delta_{a - \frac{1}{2}}} \phi(y) \rd y - \frac{\phi(\delta_a)}{\sigma_{m,x}} \left\{(c - \tfrac{1}{2}) + \frac{(c^2 - \tfrac{1}{4}) \delta_a}{2 \sigma_{m,x}}\right\}\right| \\
                &\leq \frac{\max_{z\in [\delta_{a - c}, \delta_{a - 1/2}]} |z^2 - 1| \, \phi(z)}{6 \, \sigma_{m,x}^3} \cdot \Big|c^3 - \frac{1}{8}\Big|.
            \end{aligned}
        \end{equation}
        Using \eqref{eq:main.decomposition.last.term.bound}, \eqref{eq:Cressie.generalization.eq.1.restated}, \eqref{eq:Cressie.generalization.eq.2} and the expression for $P_{k,m}(x)$ in Lemma~\ref{lem:LLT.Binomial} when $k$ is in the bulk $B_{m,x}$, the right-hand side of \eqref{eq:Cressie.generalization.eq.3} is equal to
        \begin{align}\label{eq:Cressie.generalization.eq.4.v.1}
            &\quad m^{-1/2} \cdot \left\{
                \begin{array}{l}
                    -\frac{1}{2} \Big[\big(\frac{1 - x}{x}\big)^{1/2} - \big(\frac{x}{1 - x}\big)^{1/2}\Big] \sum_{\substack{a \leq k \leq m \\ k\in B_{m,x}}} \delta_k \, \frac{\phi(\delta_k)}{\sigma_{m,x}} \\
                    + \frac{1}{6} \Big[x \big(\frac{1 - x}{x}\big)^{3/2} - (1 - x) \big(\frac{x}{1 - x}\big)^{3/2}\Big] \sum_{\substack{a \leq k \leq m \\ k\in B_{m,x}}} \delta_k^3 \, \frac{\phi(\delta_k)}{\sigma_{m,x}} \\[-1mm]
                    - (c - \tfrac{1}{2}) \frac{\phi(\delta_a)}{\sigma_x}
                \end{array}
                \right\} \notag \\
            &+ m^{-1} \cdot \left\{
            \begin{array}{l}
                \frac{1}{8} \Big[3 \big(\frac{1 - x}{x}\big) - 2 + 3 \big(\frac{x}{1 - x}\big)\Big] \sum_{\substack{a \leq k \leq m \\ k\in B_{m,x}}} \delta_k^2 \, \frac{\phi(\delta_k)}{\sigma_{m,x}} \\[1.5mm]
                -\frac{1}{12} \Big[2 x \big(\frac{1 - x}{x}\big)^2 - 1 + 2 (1 - x) \big(\frac{x}{1 - x}\big)^2\Big] \sum_{\substack{a \leq k \leq m \\ k\in B_{m,x}}} \delta_k^4 \, \frac{\phi(\delta_k)}{\sigma_{m,x}} \\[1.5mm]
                + \frac{1}{72} \Big[x^2 \big(\frac{1 - x}{x}\big)^3 - 2 \sigma_x^2 + (1 - x)^2 \big(\frac{x}{1 - x}\big)^3\Big] \sum_{\substack{a \leq k \leq m \\ k\in B_{m,x}}} \delta_k^6 \, \frac{\phi(\delta_k)}{\sigma_{m,x}} \\[-0.5mm]
                + \frac{1}{12} \big(1 - \frac{1}{x} - \frac{1}{1 - x}\big) \sum_{\substack{a \leq k \leq m \\ k\in B_{m,x}}} \frac{\phi(\delta_k)}{\sigma_{m,x}} \\[3mm]
                - \frac{1}{24} \, \frac{1}{\sigma_x^2} \sum_{\substack{a \leq k \leq m \\ k\in B_{m,x}}} (\delta_k^2 - 1) \, \frac{\phi(\delta_k)}{\sigma_{m,x}} \\[3.5mm]
                - \frac{1}{2} (c^2 - \tfrac{1}{4}) \frac{\delta_a \phi(\delta_a)}{\sigma_x^2}
            \end{array}
            \right\} \notag \\
            &+ F_m + \sqrt{\frac{8 \tau}{9 \pi m^{1/3}}} \exp\Big(-\frac{9 m^{1/3}}{32 \tau}\Big),
        \end{align}
        where the error $F_m$ satisfies
        \begin{align}
            |F_m|
            &\leq m^{-3/2} \cdot 12 \tau^5 \sum_{\substack{a \leq k \leq m \\ k\in B_{m,x}}} (1 + |\delta_k|^{12}) \frac{\phi(\delta_k)}{\sigma_{m,x}} \notag \\[-1mm]
            &\quad+ m^{-3/2} \cdot \frac{\max_{z\in [\delta_{a - c}, \delta_{a - 1/2}]} |z^2 - 1| \, \phi(z)}{6 \, \sigma_x^3} \cdot \Big|c^3 - \frac{1}{8}\Big| \notag \\
            &\quad+ m^{-2} \cdot \frac{1}{1920 \, \sigma_x^4} \sum_{\substack{a \leq k \leq m \\ k\in B_{m,x}}} \frac{\max_{z\in [\delta_{k - 1/2}, \delta_{k + 1/2}]} |z^4 - 6 z^2 + 3| \, \phi(z)}{\sigma_{m,x}}.
        \end{align}
        By Taylor expansions with Lagrange error bounds, note that
        \begin{equation}\label{eq:Cressie.generalization.eq.4.v.1.put.in.eq.1}
            \begin{aligned}
                \left|\phi(\delta_a) - \left\{\phi(\delta_{\tilde{a}}) - \frac{\delta_{\tilde{a}} \phi(\delta_{\tilde{a}})}{2 \sigma_{m,x}}\right\}\right|
                &\leq \frac{\max_{z\in [\delta_{\tilde{a}}, \delta_a]} |z^2 - 1| \, \phi(z)}{8 \sigma_{m,x}^2}, \\[1mm]
                \big|\delta_a \phi(\delta_a) - \delta_{\tilde{a}} \phi(\delta_{\tilde{a}})\big|
                &\leq \frac{\max_{z\in [\delta_{\tilde{a}}, \delta_a]} |z^2 - 1| \, \phi(z)}{2 \sigma_{m,x}},
            \end{aligned}
        \end{equation}
        where $\tilde{a} \leqdef a - \frac{1}{2}$.
        Also, the functions $y\mapsto |y|^j \phi(y)$ are monotonous by parts for all $j\in \{1,2,3,4,6,12\}$ (4 parts), so we can break off the sums $\sum_{k \geq k^{\star}} |\delta_k|^j \frac{\phi(\delta_k)}{\sigma_{m,x}}$ into 4 parts and then bound each of them with either an left or right Riemann sum on $[\delta_{k^{\star}},\infty)$ (the overlaps are just a safety measure around the three points where the graph changes direction).

        Therefore, for all $j\in \{1,2,3,4,6,12\}$,
        \begin{align}\label{eq:Cressie.generalization.eq.4.v.1.put.in.eq.2}
            \sum_{\substack{a \leq k < \infty \\ k\not\in B_{m,x}}} |\delta_k|^j \frac{\phi(\delta_k)}{\sigma_{m,x}}
            &\leq 4 \int_{\delta_{k^{\star}}}^{\infty} |y|^j \phi(y) \rd y \notag \\[-6mm]
            &\leq (\phi(\delta_{k^{\star}}) + \Psi(\delta_{k^{\star}})) \cdot
                \begin{cases}
                    4, &\mbox{if } j = 1, \\
                    4 (1 + |\delta_{k^{\star}}|), &\mbox{if } j = 2, \\
                    12 (1 + |\delta_{k^{\star}}|^2), &\mbox{if } j = 3, \\
                    16 (1 + |\delta_{k^{\star}}|^3), &\mbox{if } j = 4, \\
                    84 (1 + |\delta_{k^{\star}}|^5), &\mbox{if } j = 6,
                \end{cases}
        \end{align}
        where the last line has been verified with \texttt{Mathematica}, and
        \begin{equation}\label{eq:bound.delta.and.phi.plus.Psi}
            \begin{aligned}
                |\delta_{k^{\star}}|
                &\stackrel{\eqref{eq:bulk}}{\leq} m^{1/6} \tau, \\
                \phi(\delta_{k^{\star}}) + \Psi(\delta_{k^{\star}})
                &\stackrel{\eqref{eq:main.decomposition.last.term.bound.k.star.alone}}{\leq} \Big(\frac{1}{\sqrt{2\pi}} + \sqrt{\frac{2 \tau}{\pi m^{1/3}}}\Big) \exp\Big(-\frac{m^{1/3}}{8 \tau}\Big) \\
                &\stackrel{\phantom{\eqref{eq:main.decomposition.last.term.bound.k.star.alone}}}{\leq} \exp\Big(-\frac{m^{1/3}}{8 \tau}\Big) \quad \text{since } m \geq 2^{3/2} \tau^3.
            \end{aligned}
        \end{equation}
        Using \eqref{eq:Cressie.generalization.eq.4.v.1.put.in.eq.1}, \eqref{eq:Cressie.generalization.eq.4.v.1.put.in.eq.2} and \eqref{eq:bound.delta.and.phi.plus.Psi} in \eqref{eq:Cressie.generalization.eq.4.v.1}, the right-hand side of \eqref{eq:Cressie.generalization.eq.3} is equal to
        \begin{align}\label{eq:Cressie.generalization.eq.4.v.2}
            &\qquad m^{-1/2} \cdot \left\{
                \begin{array}{l}
                    -\frac{1}{2} \Big[\big(\frac{1 - x}{x}\big)^{1/2} - \big(\frac{x}{1 - x}\big)^{1/2}\Big] \sum_{k=a}^{\infty} \delta_k \, \frac{\phi(\delta_k)}{\sigma_{m,x}} \\
                    + \frac{1}{6} \Big[x \big(\frac{1 - x}{x}\big)^{3/2} - (1 - x) \big(\frac{x}{1 - x}\big)^{3/2}\Big] \sum_{k=a}^{\infty} \delta_k^3 \, \frac{\phi(\delta_k)}{\sigma_{m,x}} \\[-0.5mm]
                    - (c - \tfrac{1}{2}) \frac{\phi(\delta_{\tilde{a}})}{\sigma_x}
                \end{array}
                \right\} \notag \\
            &\quad+ m^{-1} \cdot \left\{
            \begin{array}{l}
                \frac{1}{8} \Big[3 \big(\frac{1 - x}{x}\big) - 2 + 3 \big(\frac{x}{1 - x}\big)\Big] \sum_{k=a}^{\infty} \delta_k^2 \, \frac{\phi(\delta_k)}{\sigma_{m,x}} \\[1.5mm]
                -\frac{1}{12} \Big[2 x \big(\frac{1 - x}{x}\big)^2 - 1 + 2 (1 - x) \big(\frac{x}{1 - x}\big)^2\Big] \sum_{k=a}^{\infty} \delta_k^4 \, \frac{\phi(\delta_k)}{\sigma_{m,x}} \\[1.5mm]
                + \frac{1}{72} \Big[x^2 \big(\frac{1 - x}{x}\big)^3 - 2 \sigma_x^2 + (1 - x)^2 \big(\frac{x}{1 - x}\big)^3\Big] \sum_{k=a}^{\infty} \delta_k^6 \, \frac{\phi(\delta_k)}{\sigma_{m,x}} \\[-0.5mm]
                + \frac{1}{12} \big(1 - \frac{1}{x} - \frac{1}{1 - x}\big) \sum_{k=a}^{\infty} \frac{\phi(\delta_k)}{\sigma_{m,x}} \\[2mm]
                - \frac{1}{24} \, \frac{1}{\sigma_x^2} \sum_{k=a}^{\infty} \delta_k^2 \, \frac{\phi(\delta_k)}{\sigma_{m,x}} + \frac{1}{24} \, \frac{1}{\sigma_x^2} \sum_{k=a}^{\infty} \frac{\phi(\delta_k)}{\sigma_{m,x}} \\[2mm]
                + \frac{1}{2} (c - \tfrac{1}{2}) \frac{\delta_{\tilde{a}} \phi(\delta_{\tilde{a}})}{\sigma_x^2} - \frac{1}{2} (c^2 - \tfrac{1}{4}) \frac{\delta_{\tilde{a}} \phi(\delta_{\tilde{a}})}{\sigma_x^2}
            \end{array}
            \right\} \notag \\
            &\quad+ \tilde{F}_m + \sqrt{\frac{8 \tau}{9 \pi m^{1/3}}} \exp\Big(-\frac{9 m^{1/3}}{32 \tau}\Big),
        \end{align}
        where the error $\tilde{F}_m$ satisfies
        \begin{align}\label{eq:F.tilde.error.bound}
            |\tilde{F}_m|
            &\leq m^{-3/2} \cdot \left\{
            \begin{array}{l}
                12 \tau^5 \cdot 4 \int_{\delta_{\tilde{a}}}^{\infty} (1 + |y|^{12}) \phi(y) \rd y \\[1mm]
                + \frac{\tau^3}{6} \big|c^3 - \frac{1}{8}\big| \cdot \max_{z\in [\delta_{a - c}, \delta_{\tilde{a}}]} |z^2 - 1| \, \phi(z) \\[1mm]
                + \frac{\tau^3}{8} |c - \tfrac{1}{2}| \cdot \max_{z\in [\delta_{\tilde{a}}, \delta_a]} |z^2 - 1| \, \phi(z) \\[1mm]
                + \frac{\tau^3}{4} |c^2 - \tfrac{1}{4}| \cdot \max_{z\in [\delta_{\tilde{a}}, \delta_a]} |z^2 - 1| \, \phi(z)
            \end{array}
            \right\} \notag \\[1mm]
            &\quad+ m^{-2} \cdot \frac{\tau^4}{1920} \sum_{k=a}^{\infty} \frac{\max_{z\in [\delta_{k - 1/2}, \delta_{k + 1/2}]} |z^4 - 6 z^2 + 3| \, \phi(z)}{\sigma_{m,x}} \notag \\[1mm]
            &\quad+
            \left\{
                \begin{array}{l}
                    m^{-1/2} \cdot 6 \tau^{3/2} \cdot (1 + |m^{1/6} \tau|^2) \\
                    + m^{-1} \cdot 14 \tau^3 \cdot (1 + |m^{1/6} \tau|^5)
                \end{array}
            \right\}
            \cdot \exp\Big(-\frac{m^{1/3}}{8 \tau}\Big).
        \end{align}

        \vspace{1mm}
        By the Euler-MacLaurin formula, we have
        \begin{equation}\label{eq:pleb.1}
            \begin{aligned}
            &\left|\sum_{k=a}^{\infty} \delta_k^j \frac{\phi(\delta_k)}{\sigma_{m,x}} -
                \left\{
                    \begin{array}{l}
                        \int_{\delta_{\tilde{a}}}^{\infty} y^j \phi(y) \rd y - \int_{\delta_{\tilde{a}}}^{\delta_a} y^j \phi(y) \rd y \\
                        + \frac{1}{2} \delta_a^j \frac{\phi(\delta_a)}{\sigma_{m,x}} - \frac{1}{12 \sigma_{m,x}^2} \big.\frac{\rd}{\rd y} y^j \phi(y)\big|_{y = \delta_a}
                    \end{array}
                \right\}
            \right| \\
            &\leq \frac{1}{12 \sigma_{m,x}^2} \int_{\delta_a}^{\infty} \Big[\frac{\rd^2}{\rd y^2} y^j \phi(y)\Big] \rd y.
            \end{aligned}
        \end{equation}
        Also, note that a Taylor expansion at $y = \delta_a$ yields
        \begin{equation}\label{eq:pleb.2}
            \begin{aligned}
                &\left|\left\{
                    \begin{array}{l}
                        - \int_{\delta_{\tilde{a}}}^{\delta_a} y^j \phi(y) \rd y \\
                        + \frac{1}{2} \delta_a^j \frac{\phi(\delta_a)}{\sigma_{m,x}} - \frac{1}{12 \sigma_{m,x}^2} \big.\frac{\rd}{\rd y} y^j \phi(y)\big|_{y = \delta_a}
                    \end{array}
                \right\} - \frac{-5}{24 \sigma_{m,x}^2} \cdot \Big.\frac{\rd}{\rd y} y^j \phi(y)\Big|_{y = \delta_a}\right| \\
                &\leq \frac{\max_{y\in \R} \big|\frac{\rd^2}{\rd y^2} y^j \phi(y)\big|}{48 \sigma_{m,x}^3} \leq \frac{1}{5 \sigma_{m,x}^3},
            \end{aligned}
        \end{equation}
        where the last inequality holds for any integer $0 \leq j \leq 6$.
        The last two equations together show that, for all $j\in \{0,1,2,3,4,6\}$,
        \begin{align}\label{eq:apply.bound}
            &\bigg|\sum_{k=a}^{\infty} \delta_k^j \frac{\phi(\delta_k)}{\sigma_{m,x}} - \int_{\delta_{\tilde{a}}}^{\infty} y^j \phi(y) \rd y\bigg| \notag \\
            &\qquad\leq \frac{1}{12 \sigma_{m,x}^2} \int_{\delta_a}^{\infty} \big\{j (j-1) |y|^{j-2} + (2j + 1) |y|^j + |y|^{j+2}\big\} \phi(y) \rd y \notag \\
            &\quad\qquad+ \frac{5}{24 \sigma_{m,x}^2} \cdot \big\{j |\delta_a|^{j-1} + |\delta_a|^{j+1}\big\} \phi(\delta_a) + \frac{1}{5 \sigma_{m,x}^3} \notag \\
            &\qquad\leq \frac{\tau^3}{5 m^{3/2}} + \frac{\tau^2}{m} \Big[\frac{1}{12} + \frac{5}{24}\Big] (\phi(\delta_a) + \Psi(\delta_a)) \cdot
                \begin{cases}
                    2 (1 + |\delta_a|), &\mbox{if } j = 0, \\
                    6 (1 + |\delta_a|^2), &\mbox{if } j = 1, \\
                    11 (1 + |\delta_a|^3), &\mbox{if } j = 2, \\
                    40 (1 + |\delta_a|^4), &\mbox{if } j = 3, \\
                    69 (1 + |\delta_a|^5), &\mbox{if } j = 4, \\
                    541 (1 + |\delta_a|^7), &\mbox{if } j = 6,
                \end{cases}
        \end{align}
        (the last line has been verified with \texttt{Mathematica}) so the right-hand side of \eqref{eq:Cressie.generalization.eq.3} is equal to
        \begin{equation}\label{eq:Cressie.generalization.eq.4.v.3}
            \begin{aligned}
            &\qquad m^{-1/2} \cdot \left\{
                \begin{array}{l}
                    -\frac{1}{2} \Big[\big(\frac{1 - x}{x}\big)^{1/2} - \big(\frac{x}{1 - x}\big)^{1/2}\Big] \int_{\delta_{\tilde{a}}}^{\infty} y \phi(y) \rd y \\
                    + \frac{1}{6} \Big[x \big(\frac{1 - x}{x}\big)^{3/2} - (1 - x) \big(\frac{x}{1 - x}\big)^{3/2}\Big] \int_{\delta_{\tilde{a}}}^{\infty} y^3 \phi(y) \rd y \\[-1mm]
                    - (c - \tfrac{1}{2}) \frac{\phi(\delta_{\tilde{a}})}{\sigma_x}
                \end{array}
                \right\} \\
            &\quad+ m^{-1} \cdot \left\{
            \begin{array}{l}
                \frac{1}{8} \Big[3 \big(\frac{1 - x}{x}\big) - 2 + 3 \big(\frac{x}{1 - x}\big)\Big] \int_{\delta_{\tilde{a}}}^{\infty} y^2 \phi(y) \rd y \\[1.5mm]
                -\frac{1}{12} \Big[2 x \big(\frac{1 - x}{x}\big)^2 - 1 + 2 (1 - x) \big(\frac{x}{1 - x}\big)^2\Big] \int_{\delta_{\tilde{a}}}^{\infty} y^4 \phi(y) \rd y \\[1.5mm]
                + \frac{1}{72} \Big[x^2 \big(\frac{1 - x}{x}\big)^3 - 2 \sigma_x^2 + (1 - x)^2 \big(\frac{x}{1 - x}\big)^3\Big] \int_{\delta_{\tilde{a}}}^{\infty} y^6 \phi(y) \rd y \\[1.5mm]
                + \frac{1}{12} \big(1 - \frac{1}{x} - \frac{1}{1 - x}\big) \int_{\delta_{\tilde{a}}}^{\infty} \phi(y) \rd y \\[1.5mm]
                - \frac{1}{24} \, \frac{1}{\sigma_x^2} \int_{\delta_{\tilde{a}}}^{\infty} y^2 \phi(y) \rd y + \frac{1}{24} \, \frac{1}{\sigma_x^2} \int_{\delta_{\tilde{a}}}^{\infty} \phi(y) \rd y \\[1.5mm]
                + \big[\frac{1}{2} (c - \tfrac{1}{2}) - \frac{1}{2} (c^2 - \tfrac{1}{4})\big] \frac{\delta_{\tilde{a}} \phi(\delta_{\tilde{a}})}{\sigma_x^2}
            \end{array}
            \right\} \\
            &\quad+ G_m,
            \end{aligned}
        \end{equation}
        where the error $G_m$ satisfies
        \begin{equation}\label{eq:error.bound.G.m}
            \begin{aligned}
                |G_m| \leq |\tilde{F}_m|
                &+ \sqrt{\frac{8 \tau}{9 \pi m^{1/3}}} \exp\Big(-\frac{9 m^{1/3}}{32 \tau}\Big) \\[1mm]
                &+ m^{-2} \cdot \tau^6 + m^{-3/2} \cdot 11 \tau^5 (1 + |\delta_a|^7) \cdot (\phi(\delta_a) + \Psi(\delta_a)),
            \end{aligned}
        \end{equation}
        after applying the bound \eqref{eq:apply.bound} in \eqref{eq:Cressie.generalization.eq.4.v.2} and noticing that our assumption $m\geq 10^3$ implies $m^{-1} \leq m^{-1/2} \cdot 10^{-3/2}$ in front of the second brace in \eqref{eq:Cressie.generalization.eq.4.v.3}.

        Since
        \begin{equation}
            \begin{aligned}
                &\int_{\delta_{\tilde{a}}}^{\infty} y \phi(y) \rd y = \phi(\delta_{\tilde{a}}), \\
                &\int_{\delta_{\tilde{a}}}^{\infty} y^3 \phi(y) \rd y = (\delta_{\tilde{a}}^2 + 2) \phi(\delta_{\tilde{a}}),
            \end{aligned}
        \end{equation}
        (these identities are easily verified with \texttt{Mathematica})
        we see that the first brace in \eqref{eq:Cressie.generalization.eq.4.v.3} is zero with the following choice of $c$ (denoted by $\tilde{c}$):
        \begin{align}
            \tilde{c}
            &\leqdef \frac{1}{2} + \left\{
                \begin{array}{l}
                    - \frac{1}{2} \Big[\big(\frac{1 - x}{x}\big)^{1/2} - \big(\frac{x}{1 - x}\big)^{1/2}\Big] \cdot 1 \\[1mm]
                    + \frac{1}{6} \Big[x \big(\frac{1 - x}{x}\big)^{3/2} - (1 - x) \big(\frac{x}{1 - x}\big)^{3/2}\Big] \cdot (\delta_{\tilde{a}}^2 + 2)
                \end{array}
                \right\} \cdot \sigma_x \notag \\
            &\hspace{0.95mm}= \frac{1}{2} + \frac{(1 - 2x)}{6} \left[\delta_{\tilde{a}}^2 - 1\right].
        \end{align}
        With that choice, now consider
        \begin{equation}\label{eq:error.terms.brace.2}
            c = \tilde{c} + \frac{w}{\sqrt{m}},
        \end{equation}
        in \eqref{eq:Cressie.generalization.eq.4.v.3}.
        The terms of order $m^{-1/2}$ cancel out and the terms of order $m^{-1}$ are :
        \begin{align}\label{eq:error.terms.brace.2.next}
            &\quad- w \, \frac{\phi(\delta_{\tilde{a}})}{\sigma_x} + \left\{
            \begin{array}{l}
                \frac{1}{8} \Big[3 \big(\frac{1 - x}{x}\big) - 2 + 3 \big(\frac{x}{1 - x}\big)\Big] \int_{\delta_{\tilde{a}}}^{\infty} y^2 \phi(y) \rd y \\[1.5mm]
                -\frac{1}{12} \Big[2 x \big(\frac{1 - x}{x}\big)^2 - 1 + 2 (1 - x) \big(\frac{x}{1 - x}\big)^2\Big] \int_{\delta_{\tilde{a}}}^{\infty} y^4 \phi(y) \rd y \\[1.5mm]
                + \frac{1}{72} \Big[x^2 \big(\frac{1 - x}{x}\big)^3 - 2 \sigma_x^2 + (1 - x)^2 \big(\frac{x}{1 - x}\big)^3\Big] \int_{\delta_{\tilde{a}}}^{\infty} y^6 \phi(y) \rd y \\[2mm]
                + \frac{1}{12} \big(1 - \frac{1}{x} - \frac{1}{1 - x}\big) \int_{\delta_{\tilde{a}}}^{\infty} \phi(y) \rd y \\[1mm]
                - \frac{1}{24} \, \frac{1}{\sigma_x^2} \int_{\delta_{\tilde{a}}}^{\infty} y^2 \phi(y) \rd y + \frac{1}{24} \, \frac{1}{\sigma_x^2} \int_{\delta_{\tilde{a}}}^{\infty} \phi(y) \rd y \\[1mm]
                + \big[\frac{1}{2} (c - \tfrac{1}{2}) - \frac{1}{2} (c^2 - \tfrac{1}{4})\big] \frac{\delta_{\tilde{a}} \phi(\delta_{\tilde{a}})}{\sigma_x^2}
            \end{array}
            \right\} \notag \\
            &= - w \, \frac{\phi(\delta_{\tilde{a}})}{\sigma_x} + \left\{
                \begin{array}{l}
                    \Big[\frac{\frac{1}{3}}{\sigma_x^2} - 1\Big] \int_{\delta_{\tilde{a}}}^{\infty} y^2 \phi(y) \rd y \\[1.5mm]
                    + \Big[\frac{-\frac{1}{6}}{\sigma_x^2} + \frac{7}{12}\Big] \int_{\delta_{\tilde{a}}}^{\infty} y^4 \phi(y) \rd y \\[1.5mm]
                    + \Big[\frac{\frac{1}{72}}{\sigma_x^2} - \frac{1}{18}\Big] \int_{\delta_{\tilde{a}}}^{\infty} y^6 \phi(y) \rd y \\[2mm]
                    + \Big[\frac{1}{12} - \frac{\frac{1}{24}}{\sigma_x^2}\Big] \int_{\delta_{\tilde{a}}}^{\infty} \phi(y) \rd y \\[1.5mm]
                    - \big[\frac{1}{2} (\tilde{c} - \tfrac{1}{2})^2 + (\tilde{c} - \tfrac{1}{2}) \frac{w}{\sqrt{m}} + \frac{w^2}{2 m}\big] \cdot \frac{\delta_{\tilde{a}} \phi(\delta_{\tilde{a}})}{\sigma_x^2}
                \end{array}
                \right\}.
        \end{align}
        Since
        \begin{equation}
            \begin{aligned}
                \int_{\delta_{\tilde{a}}}^{\infty} y^2 \phi(y) \rd y
                &= \delta_{\tilde{a}} \phi(\delta_{\tilde{a}}) + \Psi(\delta_{\tilde{a}}), \\
                \int_{\delta_{\tilde{a}}}^{\infty} y^4 \phi(y) \rd y
                &= \big[\delta_{\tilde{a}}^3 + 3 \delta_{\tilde{a}}\big] \phi(\delta_{\tilde{a}}) + 3 \Psi(\delta_{\tilde{a}}), \\
                \int_{\delta_{\tilde{a}}}^{\infty} y^6 \phi(y) \rd y
                &= \big[\delta_{\tilde{a}}^5 + 5 \delta_{\tilde{a}}^3 + 15 \delta_{\tilde{a}}\big] \phi(\delta_{\tilde{a}}) + 15 \Psi(\delta_{\tilde{a}}),
            \end{aligned}
        \end{equation}
        (again, these identities are easily verified with \texttt{Mathematica}), then all the terms with a $\Psi(\delta_{\tilde{a}})$ factor cancel out in \eqref{eq:error.terms.brace.2.next} and the expression simplifies to
        \begin{align}\label{eq:error.terms.brace.2.next.simplified}
            &\quad- w \, \frac{\phi(\delta_{\tilde{a}})}{\sigma_x} + \left\{
            \begin{array}{l}
                \Big[\frac{\frac{1}{3}}{\sigma_x^2} - 1\Big] \cdot \delta_{\tilde{a}} \phi(\delta_{\tilde{a}}) \\[1.5mm]
                + \Big[\frac{-\frac{1}{6}}{\sigma_x^2} + \frac{7}{12}\Big] \cdot \big[\delta_{\tilde{a}}^3 + 3 \delta_{\tilde{a}}\big] \phi(\delta_{\tilde{a}}) \\[1.5mm]
                + \Big[\frac{\frac{1}{72}}{\sigma_x^2} - \frac{1}{18}\Big] \cdot \big[\delta_{\tilde{a}}^5 + 5 \delta_{\tilde{a}}^3 + 15 \delta_{\tilde{a}}\big] \phi(\delta_{\tilde{a}}) \\[2mm]
                - \frac{1}{72} (1 - 2x)^2 \left[\delta_{\tilde{a}}^2 - 1\right]^2 \cdot \frac{\delta_{\tilde{a}} \phi(\delta_{\tilde{a}})}{\sigma_x^2}
            \end{array}
            \right\} \notag \\
            &\quad- m^{-1/2} \cdot \left\{\frac{w}{6} (1 - 2 x) [\delta_{\tilde{a}}^2 - 1] \cdot \frac{\delta_{\tilde{a}} \phi(\delta_{\tilde{a}})}{\sigma_x^2}\right\} - m^{-1} \cdot \left\{\frac{w^2}{2} \cdot \frac{\delta_{\tilde{a}} \phi(\delta_{\tilde{a}})}{\sigma_x^2}\right\} \notag \\
            &= - w \, \frac{\phi(\delta_{\tilde{a}})}{\sigma_x} + \left\{
            \begin{array}{l}
                \Big[\frac{\frac{1}{3}}{\sigma_x^2} - 1 + \frac{-\frac{3}{6}}{\sigma_x^2} + \frac{21}{12} + \frac{\frac{15}{72}}{\sigma_x^2} - \frac{15}{18} + \frac{-\frac{1}{72}}{\sigma_x^2} + \frac{1}{18}\Big] \cdot \delta_{\tilde{a}} \, \phi(\delta_{\tilde{a}}) \\[2mm]
                + \Big[\frac{-\frac{1}{6}}{\sigma_x^2} + \frac{7}{12} + \frac{\frac{5}{72}}{\sigma_x^2} - \frac{5}{18} + \frac{\frac{2}{72}}{\sigma_x^2} - \frac{2}{18}\Big] \cdot \delta_{\tilde{a}}^3 \, \phi(\delta_{\tilde{a}}) \\[2mm]
                + \Big[\frac{\frac{1}{72}}{\sigma_x^2} - \frac{1}{18} + \frac{-\frac{1}{72}}{\sigma_x^2} + \frac{1}{18}\Big] \cdot \delta_{\tilde{a}}^5 \, \phi(\delta_{\tilde{a}})
            \end{array}
            \right\} \notag \\
            &\quad- m^{-1/2} \cdot \left\{\frac{w}{6} (1 - 2 x) [\delta_{\tilde{a}}^2 - 1] \cdot \frac{\delta_{\tilde{a}} \phi(\delta_{\tilde{a}})}{\sigma_x^2}\right\} - m^{-1} \cdot \left\{\frac{w^2}{2} \cdot \frac{\delta_{\tilde{a}} \phi(\delta_{\tilde{a}})}{\sigma_x^2}\right\} \notag \\
            &= - w \, \frac{\phi(\delta_{\tilde{a}})}{\sigma_x} + \left\{\Big[\frac{\frac{1}{36}}{\sigma_x^2} - \frac{1}{36}\Big] \cdot \delta_{\tilde{a}} \, \phi(\delta_{\tilde{a}}) + \Big[\frac{-\frac{5}{72}}{\sigma_x^2} + \frac{7}{36}\Big] \cdot \delta_{\tilde{a}}^3 \, \phi(\delta_{\tilde{a}})\right\} \notag \\
            &\quad- m^{-1/2} \cdot \left\{\frac{w}{6} (1 - 2 x) [\delta_{\tilde{a}}^2 - 1] \cdot \frac{\delta_{\tilde{a}} \phi(\delta_{\tilde{a}})}{\sigma_x^2}\right\} - m^{-1} \cdot \left\{\frac{w^2}{2} \cdot \frac{\delta_{\tilde{a}} \phi(\delta_{\tilde{a}})}{\sigma_x^2}\right\}.
        \end{align}
        Now, we see that the second to last line is zero for the following choice of $w$ (denoted by $\tilde{w}$):
        \begin{align}
            \tilde{w}
            &\leqdef \bigg\{\Big[\frac{\frac{1}{36}}{\sigma_x^2} - \frac{1}{36}\Big] \cdot \delta_{\tilde{a}} + \Big[\frac{-\frac{5}{72}}{\sigma_x^2} + \frac{7}{36}\Big] \cdot \delta_{\tilde{a}}^3\bigg\} \cdot \sigma_x.
        \end{align}
        Therefore, the choice
        \begin{equation}
            \begin{aligned}
                c_{m,x}^{\star}(a) \leqdef \tilde{c} + \frac{\tilde{w}}{\sqrt{m}}
                &= \frac{1}{2} + \frac{(1 - 2x)}{6} \left[\delta_{\tilde{a}}^2 - 1\right] \\
                &\quad+ \frac{1}{\sqrt{m}} \cdot \frac{1}{\sigma_x} \left\{\Big[\frac{1}{36} - \frac{\sigma_x^2}{36}\Big] \cdot \delta_{\tilde{a}} + \Big[-\frac{5}{72} + \frac{7 \sigma_x^2}{36}\Big] \cdot \delta_{\tilde{a}}^3\right\},
            \end{aligned}
        \end{equation}
        guarantees that the terms of order $m^{-1/2}$ and $m^{-1}$ in \eqref{eq:Cressie.generalization.eq.4.v.3} are all zero.
        We can also control the last line of \eqref{eq:error.terms.brace.2.next.simplified}:
        \begin{equation}
            \left|
            \begin{array}{l}
                m^{-1/2} \cdot \left\{\frac{\tilde{w}}{6} (1 - 2 x) [\delta_{\tilde{a}}^2 - 1] \cdot \frac{\delta_{\tilde{a}} \phi(\delta_{\tilde{a}})}{\sigma_x^2}\right\} \\
                + \, m^{-1} \cdot \left\{\frac{\tilde{w}^2}{2} \cdot \frac{\delta_{\tilde{a}} \phi(\delta_{\tilde{a}})}{\sigma_x^2}\right\}
            \end{array}
            \right| \leq m^{-1/2} \cdot 10 \tau^4 (1 + |\delta_{\tilde{a}}|^6) \cdot \phi(\delta_{\tilde{a}}).
        \end{equation}
        Putting this error bound in \eqref{eq:Cressie.generalization.eq.4.v.3} with the choice $c = c_{m,x}^{\star}(a)$, we get that the right-hand side of \eqref{eq:Cressie.generalization.eq.3} is equal to $\tilde{G}_m$, where
        \begin{equation}\label{eq:complete.bound.CC}
            \begin{aligned}
                |\tilde{G}_m| \leq |\tilde{F}_m|
                &+ \sqrt{\frac{8 \tau}{9 \pi m^{1/3}}} \exp\Big(-\frac{9 m^{1/3}}{32 \tau}\Big) \\
                &+ m^{-2} \cdot \tau^6 + m^{-3/2} \cdot \left\{
                \begin{array}{l}
                    11 \tau^5 (1 + |\delta_a|^7) \cdot (\phi(\delta_a) + \Psi(\delta_a)) \\[1mm]
                    + 10 \tau^4 (1 + |\delta_{\tilde{a}}|^6) \cdot \phi(\delta_{\tilde{a}})
                \end{array}
                \right\}.
            \end{aligned}
        \end{equation}
        Since
        \begin{equation}
            \begin{aligned}
                |c_{m,x}^{\star}(a) - \tfrac{1}{2}| &\leq \tau (1 + |\delta_{\tilde{a}}|^3), \\
                |(c_{m,x}^{\star}(a))^2 - \tfrac{1}{4}| &\leq 4 \tau^2 (1 + |\delta_{\tilde{a}}|^6), \\
                |(c_{m,x}^{\star}(a))^3 - \tfrac{1}{8}| &\leq 10 \tau^3 (1 + |\delta_{\tilde{a}}|^9),
            \end{aligned}
        \end{equation}
        a more explicit bound on $|\tilde{G}_m|$ is:
        \begin{equation}\label{eq:near.end.survival.estimate}
            \begin{aligned}
                |\tilde{G}_m|
                &\stackrel{\eqref{eq:F.tilde.error.bound}}{\leq} m^{-3/2} \cdot \left\{
                \begin{array}{l}
                    12 \tau^5 \cdot 4 \int_{\delta_{\tilde{a}}}^{\infty} (1 + |y|^{12}) \phi(y) \rd y \\[1mm]
                    + 3 \tau^4 \cdot \max_{z\in \R} |z^2 - 1| \, \phi(z) \\[1mm]
                    + 11 \tau^5 \cdot \max_{z\in [0, \infty)} \big\{ (1 + |z|^7) (\phi(z) + \Psi(z))\big\} \\[1mm]
                    + 10 \tau^4 \cdot \max_{z\in [-\frac{1}{2 \sigma_{m,x}}, \infty)} \big\{(1 + |z|^6) \phi(z)\big\}
                \end{array}
                \right\} \\[1mm]
                &\qquad+ m^{-2} \cdot \left\{\tau^6 + \frac{\tau^4}{1920} \sum_{k=a}^{\infty} \frac{\max_{z\in [\delta_{k - 1/2}, \delta_{k + 1/2}]} |z^4 - 6 z^2 + 3| \, \phi(z)}{\sigma_{m,x}}\right\} \\
                &\qquad+ m^{-1/6} \cdot \left\{\left\{20 \tau^3 (1 + \tau^5)\right\}
                \cdot \exp\Big(-\frac{m^{1/3}}{8 \tau}\Big) + \sqrt{\frac{8 \tau}{9 \pi}} \exp\Big(-\frac{9 m^{1/3}}{32 \tau}\Big)\right\},
            \end{aligned}
        \end{equation}
        where we used the assumption that $a \geq m x$ to obtain that $\delta_{\tilde{a}} \geq -\frac{1}{2 \sigma_{m,x}}$ and thus restrict the domain of the last maximum inside the first brace in \eqref{eq:near.end.survival.estimate} to $[-\frac{1}{2 \sigma_{m,x}},\infty)$.
        Note that $m \geq 2^{3/2} \tau^3$ easily implies that $[-\frac{1}{2 \sigma_{m,x}},\infty)\subseteq [-1,\infty)$ and we can verify with \texttt{Mathematica} that
        \begin{equation}\label{eq:final.bound.eq.1}
            \begin{aligned}
                \max_{z\in [0, \infty)} \big\{ (1 + |z|^7) (\phi(z) + \Psi(z))\big\} &\leq 15, \\
                \max_{z\in [-1, \infty)} \big\{(1 + |z|^6) \phi(z)\big\} &\leq 5.
            \end{aligned}
        \end{equation}
        Also, the other maximum in the first brace of \eqref{eq:near.end.survival.estimate} can be bounded by noticing that
        \begin{equation}\label{eq:final.bound.eq.2}
            \max_{z\in \R} \big\{ |z^2 - 1| \phi(z)\big\} = \phi(0) = \frac{1}{\sqrt{2\pi}} \leq 1,
        \end{equation}
        Finally, we have
        \begin{equation}\label{eq:final.bound.eq.3}
            \begin{aligned}
                \sum_{k=a}^{\infty} \frac{\max_{z\in [\delta_{k - 1/2}, \delta_{k + 1/2}]} |z^4 - 6 z^2 + 3| \, \phi(z)}{\sigma_{m,x}}
                &\leq 4 \int_{\delta_{\tilde{a}}}^{\infty} (1 + |y|^{12}) \phi(y) \rd y \\
                &\leq 20 \, 800,
            \end{aligned}
        \end{equation}
        uniformly for $\tilde{a}\in [-1,\infty)$, and with much room to spare for the first inequality.
        Applying \eqref{eq:final.bound.eq.1}, \eqref{eq:final.bound.eq.2} and \eqref{eq:final.bound.eq.3} together in \eqref{eq:near.end.survival.estimate} give us the much simpler bound:
        \begin{equation}\label{eq:near.end.survival.estimate.next}
            \begin{aligned}
                |\tilde{G}_m|
                &\leq m^{-3/2} \cdot 250 \, 000 \cdot \tau^5 + m^{-1/6} \cdot 41 \tau^8 \cdot \exp\Big(-\frac{m^{1/3}}{8 \tau}\Big) \\
                &\leq m^{-3/2} \cdot 10^6 \tau^5, \qquad \text{by the second condition in \eqref{eq:survival.estimate.official.conditions.on.m}}.
            \end{aligned}
        \end{equation}
        To be clear, almost all the contribution to the $250 \, 000$ on the first line of \eqref{eq:near.end.survival.estimate.next} comes from $12 \tau^5 \cdot 4 \int_{\delta_{\tilde{a}}}^{\infty} (1 + |y|^{12}) \phi(y) \rd y$ in \eqref{eq:near.end.survival.estimate}.
        This ends the proof.
    \end{proof}

    We are now ready to prove the most important result of the paper, which improves on the versions of Tusn\'ady's inequality from \cite{MR1955348} and \cite{MR2154001} in the bulk, see Remark~\ref{rem:Tusnady.improvement}.

    \begin{proof}[Proof of Theorem~\ref{thm:Tusnady.inequality.improved}]
        Note that $s\mapsto (1 - \frac{1}{12 m}) s + \frac{1}{3 m^2} s^3$ is increasing, so Corollary~\ref{cor:refined.CC.x.1.2} yields
        \begin{equation}\label{eq:thm:Tusnady.inequality.improved.eq.1}
            \Phi\bigg(\frac{(1 - \frac{1}{12 m}) \tilde{s} + \frac{1}{3 m^2} \tilde{s}^3}{\sqrt{m} / 2}\bigg) \leq \PP(X_m \leq \tfrac{m}{2} + t) + 10^6 2^5 m^{-3/2},
        \end{equation}
        with $\tilde{s} \leqdef t - 1/2$ and $|t| \leq \frac{m^{2/3}}{2} - 1$.
        If we assume further that $|t| \leq \frac{1}{4} \sqrt{m \log m}$, then the mean value theorem yields
        \begin{equation}\label{eq:thm:Tusnady.inequality.improved.eq.2}
            10^6 2^5 m^{-3/2} \leq \Phi\bigg(\frac{(1 - \frac{1}{12 m}) \tilde{s} + \frac{1}{3 m^2} \tilde{s}^3}{\sqrt{m} / 2}\bigg) - \Phi(z),
        \end{equation}
        with
        \begin{equation}\label{eq:def.z}
            z \leqdef \frac{(1 - \frac{1}{12 m}) \tilde{s} + \frac{1}{3 m^2} \tilde{s}^3}{\sqrt{m} / 2} - \frac{10^6 2^5 m^{-3/2}}{\phi(\sqrt{\log m})},
        \end{equation}
        since $\Phi' = \phi$, and our assumptions $|t| \leq \frac{1}{4} \sqrt{m \log m}$ and {\color{black} $\sqrt{2\pi} \, 20^6 m^{-1} \leq \sqrt{\log m}$} together clearly imply that
        \begin{equation}\label{eq:interval.contains.MVT}
            \bigg[\frac{(1 - \frac{1}{12 m}) \tilde{s} + \frac{1}{3 m^2} \tilde{s}^3}{\sqrt{m} / 2}, \, z\bigg]\subseteq [-\sqrt{\log m}, \sqrt{\log m}].
        \end{equation}
        By plugging \eqref{eq:thm:Tusnady.inequality.improved.eq.2} into \eqref{eq:thm:Tusnady.inequality.improved.eq.1}, we get
        \begin{equation}\label{eq:thm:Tusnady.inequality.improved.eq.3}
            \Phi(z) \leq \PP(X_m \leq \tfrac{m}{2} + t).
        \end{equation}
        We can rewrite \eqref{eq:def.z} as
        \begin{equation}\label{eq:cubic.equation}
            \frac{1}{3 m^2} \tilde{s}^3 + \Big(1 - \frac{1}{12 m}\Big) \tilde{s} = \tilde{z}, \quad \text{where } \tilde{z} \leqdef \frac{\sqrt{m}}{2} \Big(z + \frac{10^6 2^5 m^{-3/2}}{\phi(\sqrt{\log m})}\Big).
        \end{equation}
        We can solve this cubic equation in $t$.
        Specifically, for a cubic equation of the form $a x^3 + c x = d$ with $a,c > 0$, the unique real solution is
        \begin{equation}
            x = \sqrt[3]{\frac{d}{2a} + \sqrt{\Big(\frac{d}{2 a}\Big)^2 + \Big(\frac{c}{3 a}\Big)^3}} + \sqrt[3]{\frac{d}{2a} - \sqrt{\Big(\frac{d}{2 a}\Big)^2 + \Big(\frac{c}{3 a}\Big)^3}}.
        \end{equation}
        Therefore, the solution $t = \frac{1}{2} + \tilde{s}$ in \eqref{eq:cubic.equation} is
        \begin{equation}\label{eq:thm:Tusnady.inequality.improved.eq.4}
            \begin{aligned}
                t
                &= \frac{1}{2} + \sqrt[3]{\frac{3 m^2 \tilde{z}}{2} + \sqrt{\Big(\frac{3 m^2 \tilde{z}}{2}\Big)^2 + \Big(m^2 - \frac{m}{12}\Big)^3}} \\
                &\qquad+ \sqrt[3]{\frac{3 m^2 \tilde{z}}{2} - \sqrt{\Big(\frac{3 m^2 \tilde{z}}{2}\Big)^2 + \Big(m^2 - \frac{m}{12}\Big)^3}}.
            \end{aligned}
        \end{equation}
        Using a Taylor expansion (for $m^{-1}$ at $0$) together with Lagrange's error bound, we have, for all $|\tilde{z}| \leq \sqrt{m \log m}$, 
        \begin{equation}\label{eq:cubic.equation.inverse.Taylor}
            \begin{aligned}
                \left|t - \left\{\frac{1}{2} + \tilde{z} + \frac{\tilde{z}}{12 m} - \frac{\tilde{z}^3}{3 m^2}\right\}\right| \leq \frac{m^2 |\tilde{z}| + m |\tilde{z}|^3 + |\tilde{z}|^5}{3 m^4} \leq \frac{(\log m)^{5/2}}{m^{3/2}}.
            \end{aligned}
        \end{equation}
        Since $\big|\frac{10^6 2^5 m^{-3/2}}{\phi(\sqrt{\log m})}\big| = \sqrt{2\pi} \cdot 10^6 2^5 m^{-1}$ and $|z| \leq \sqrt{\log m}$ by \eqref{eq:interval.contains.MVT}, Equations \eqref{eq:thm:Tusnady.inequality.improved.eq.3} and \eqref{eq:cubic.equation.inverse.Taylor} together yield
        \begin{align}\label{eq:thm:Tusnady.inequality.improved.eq.5}
            \Phi(z)
            &\leq \PP\bigg(X_m \leq \frac{m}{2} + \frac{1}{2} + \Big(1 + \frac{1}{12 m}\Big) \tilde{z}  - \frac{\tilde{z}^3}{3 m^2} + \frac{(\log m)^{5/2}}{m^{3/2}}\bigg), \notag \\
            &\leq \PP\bigg(X_m \leq \frac{m}{2} + \frac{1}{2} + \frac{\sqrt{m}}{2} z + \frac{z - z^3}{24 \sqrt{m}} + \frac{2 \cdot 20^6}{m}\bigg).
        \end{align}
        By applying $F^{\star}$ on both sides, we obtain
        \begin{equation}
            X_m - \frac{m}{2} - \frac{\sqrt{m}}{2} Z - \frac{Z - Z^3}{24 \sqrt{m}}  \leq  \frac{1}{2} + \frac{2 \cdot 20^6}{m}, \quad \text{for } |Z| \leq \sqrt{\log m}.
        \end{equation}
        The conclusion follows by the symmetry of $X_m - \frac{m}{2}$ and $Z$.
    \end{proof}

\section*{Data Availability Statement}

    Data sharing not applicable to this article as no datasets were generated or analysed during the current study.

%
%

\section*{References}
\phantomsection
\addcontentsline{toc}{chapter}{References}

\bibliographystyle{authordate1}
\bibliography{Ouimet_2021_improvement_Tusnady_bib}

\end{document}